\documentclass[12pt]{amsart}

\usepackage[left=3.5cm, right=3.5cm, top=3cm]{geometry}
\usepackage{color,amssymb,graphicx, amsmath, amsthm, enumitem,MnSymbol, array,mathalpha,tikz-cd}
\usepackage{tikz}

\DeclareMathAlphabet\mathbfcal{OMS}{cmsy}{b}{n}

\usepackage{floatflt}

\usepackage{hyperref}

\hypersetup{
  pdftitle={The Nash manifold of four-point configurations modulo similarity subgroups},
  pdfauthor={Bruce Olberding and Elaine A. Walker},
  pdfkeywords={Moduli space, point configuration, quadrilateral, quadrangle, nine-point conic}
}

\usepackage{mathbbol}

 \usepackage[mathscr]{eucal}

\usepackage[font=small]{caption}
\newcommand{\R}{\mathbb{R}}
\newcommand{\Q}{{\mathcal Q}}
\newcommand{\Quad}{{\mathcal Q}}
\renewcommand{\P}{{\mathbb{RP}}}

\DeclareMathOperator{\sgn}{sgn}
\DeclareMathOperator{\sph}{\mathrm{sph}}
\DeclareMathOperator{\tr}{\mathsf{tr}}
\begin{document}

\newtheorem{theorem}{Theorem}[section]
\newtheorem{lemma}[theorem]{Lemma}
\newtheorem{proposition}[theorem]{Proposition}
\newtheorem{corollary}[theorem]{Corollary}
\newtheorem{problem}[theorem]{Problem}
\newtheorem{construction}[theorem]{Construction}

\theoremstyle{definition}
\newtheorem{defi}[theorem]{Definitions}
\newtheorem{definition}[theorem]{Definition}
\newtheorem{notation}[theorem]{Notation}
\newtheorem{remark}[theorem]{Remark}
\newtheorem{example}[theorem]{Example}
\newtheorem{question}[theorem]{Question}
\newtheorem{comment}[theorem]{Comment}
\newtheorem{comments}[theorem]{Comments}

\newtheorem{discussion}[theorem]{Discussion}

\renewcommand{\thedefi}{}

\parskip=.5mm

\title[The Nash manifold of four-point configurations]{The Nash manifold of four-point configurations modulo similarity subgroups}

\author{Bruce Olberding}
\address{Department of Mathematical Sciences, New Mexico State University\newline Las Cruces, NM 88003-8001}

\email{bruce@nmsu.edu}

\author{Elaine A.~Walker}
\address{Las Cruces, NM}

\email{elaine@ziaspace.com}

\begin{abstract}  We construct and study the Nash manifold of four-point configurations in the real plane modulo the action of a Nash subgroup of the group of similarity transformations. We do so by using the finer notion of a quadrangle in place of that of a four-point configuration, since this retains limiting line data in degenerations. We prove that the space $\Q$ of quadrangles is an 8-dimensional Nash manifold and that, for every Nash subgroup $G$ of the similarity group, the orbit space $\Q/G$ is a Nash 
manifold. We  define a geometric invariant, called aspect, with values in $[-1,1]$, and prove that, over each of the intervals $(-1,0)$ and $(0,1)$, the corresponding part of $\Q/G$ is Nash diffeomorphic to the product of the interval with a fixed fiber. Thus the nonexceptional part of the moduli problem reduces to the analysis of two model fibers, each having a natural geometric interpretation in terms of the quadrangles themselves.

 \end{abstract}



\date{\today}

\maketitle

\section{Introduction}

In this article we  construct a Nash moduli space for four-point configurations in the real plane by replacing configurations of points with quadrangles, which are four-point configurations that keep their line data in degenerations, and we use the nine-point conic of a quadrangle to define an invariant that organizes the resulting moduli space. Closely related moduli spaces of quadrilaterals have been studied extensively; see, for example,  
\cite{GN,KapMil1,KapMil2, ShiVan,StClair}, where the structure of moduli spaces of fixed-side-length quadrilaterals is described, \cite{GonLop1,GonLop2, GonMen}, where moduli spaces of simple quadrilaterals are examined, or  \cite[Section~4.2]{ArrMuc}, where generic unordered four-point configurations are considered modulo the full affine group. A basic issue in dealing with such  moduli spaces is how to handle degenerations. For example,    {Figure~1 illustrates how convergence of four vertices alone need not control the limiting slopes of the joining lines, or numerical invariants that depend on them.}
     \begin{figure}[ht]
     \begin{center}
\includegraphics[width=.82\textwidth,scale=.03]{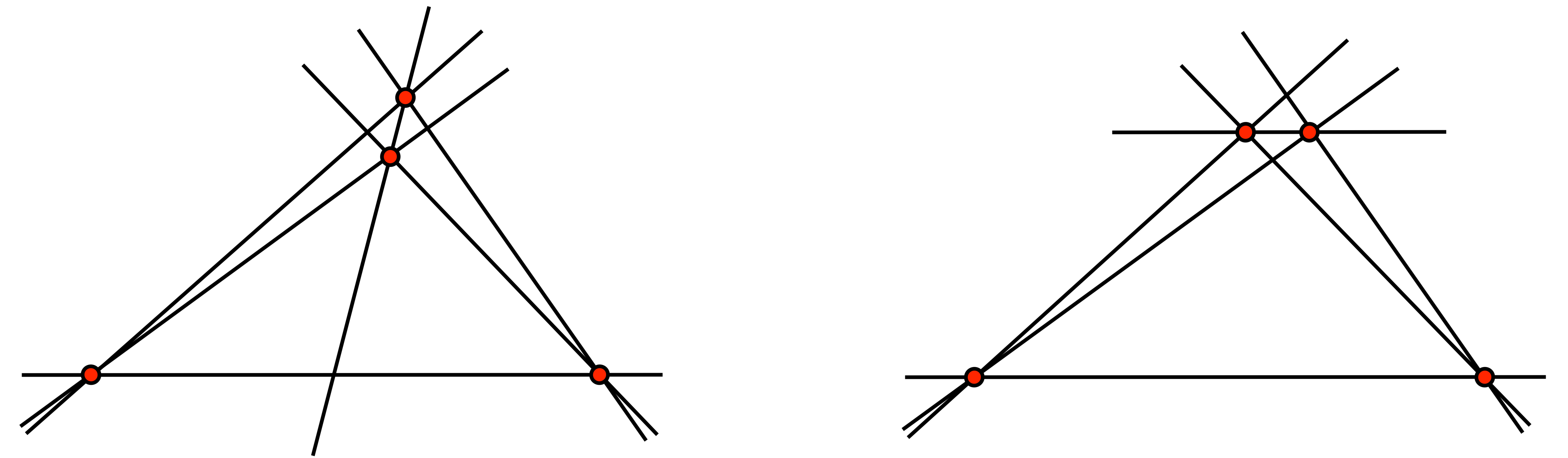}
 \end{center}
\caption{Ignoring the lines, these two configurations of points can be made arbitrarily close, after translating one of them, by deforming them so that the top two vertices in each approach one another. However, the slopes of the lines joining the top pair in the two configurations do not approach one another, so when the lines are included in the data of the configuration, these configurations remain far apart. }
\end{figure}

To retain this data that is otherwise lost in degenerations, we enrich the notion of a four-point configuration and consider the associated {\it quadrangle}, a configuration consisting of four vertices together with  six lines (``sides'') joining each pair of vertices. {We allow degenerations but require at least three sides to extend those of a nondegenerate triangle.} While this introduction of quadrangles into the picture adds no new information in the {nondegenerate} case, it remembers line data in degenerations that would otherwise be lost, and it leads to a finer notion that can be viewed as preventing some ``wormholes'' in the moduli space.

Unlike the usual spaces of quadrilateral shapes, which encode only the vertices modulo a chosen transformation group, our 
construction also retains line data in degenerations. In this respect, it is reminiscent of Fulton--MacPherson-type compactifications of configuration spaces \cite{FM}, which retain information about nested coincidence of points, relative scales, and limiting directions. By contrast, our space of quadrangles retains the four vertices together with the six pairwise projective line directions, subject to the requirement that at least three of the corresponding lines extend the sides of a nondegenerate triangle. This restricted treatment of degenerations is tailored to four-point geometry and to quotients by similarity subgroups.

The space $\Q$ of quadrangles has a well-behaved geometric structure, but this is not for free. The proof requires the {nondegenerate} triangular core in the definition of quadrangle, as does the properness of the similarity actions considered in the article. The first main result is then that the quadrangle construction {defines} a genuine Nash moduli space rather than merely a set of equivalence classes. We prove that $\Q$ is an 8-dimensional Nash manifold and that, for every Nash subgroup $G\subseteq \operatorname{Sim}(\R^2)$ (such as the oriented or unoriented isometry and similarity groups), the quotient $\Q/G$ is a Nash manifold of dimension $8-\dim G$, with quotient map a Nash submersion and a principal $G$-bundle.

This result provides the setting for the rest of the paper. We use {a classical construction from nineteenth-century geometry}, the nine-point conic of a quadrangle, to put the moduli space into a more tractable form. Using the shape data of the nine-point conic, we define what we call a form map from $\Q$ to $\P^2$ and prove it  is a Nash submersion. Using the form map, we define sphericity and aspect functions on $\Q$. Sphericity is Nash on all of $\Q$, whereas aspect is Nash on the open subset consisting of quadrangles with aspect in $(-1,1)$, but its restriction to this subset has no Nash extension to all of $\Q$. Nevertheless, its level sets {are} Nash submanifolds. In particular, for every aspect $\alpha\in[-1,1]$, {the orbit space of aspect $\alpha$ quadrangles} is a closed embedded Nash submanifold of $\Q/G$, of codimension $1$ except at aspect $1$, where the codimension is $2$. When the aspect is not $-1$, $0$ or $1$, aspect becomes the organizing parameter for the quotient. If $\alpha,\beta\in(-1,0)\cup(0,1)$ have the same sign, then the aspect $\alpha$ and aspect $\beta$ level sets in $\Q$ are $G$-equivariantly Nash diffeomorphic. This is another main result of the paper, and it is contained in Theorem~\ref{hard}, which constructs a global, similarity-equivariant Nash ``boost'' between the aspect levels. Consequently, the elliptic and hyperbolic 
nonexceptional pieces of $\Q/G$ are Nash-trivial over intervals, with typical fibers the aspect $1/2$ and aspect $-1/2$ orbit spaces, respectively.

 \vspace{.1in}

{\it Notation.} We denote by $\R$ the set of real numbers, and  by $\P^1$ and $\P^2$ the real projective line and  real projective plane, respectively.

\section{Quadrangles}

A  (complete) quadrangle as classically defined is a planar configuration of four points in general position and the six lines through pairs of the four points. We loosen this definition in order to permit degenerations. Permitting all possible degenerations results is what we will call a vertex configuration.\footnote{{ In Martin's terminology \cite[Definition~3.1]{Martin}, our vertex configurations are real pictures of the complete graph $K_4$, with all vertices in the affine plane.}}

 \begin{definition} \label{quadrangle def}

A {\it (labeled) vertex configuration} $Q$ consists of four points (called {\it vertices}) given in order $v_1,\ldots,v_4$ and six lines ({\it sides}), ${L_{12},L_{13},L_{14},L_{23},L_{24},L_{34}}$. 
For $1 \leq i < j \leq 4$, the vertices $v_i$ and $v_j$ lie on ${L_{ij}}$. A point at which exactly two vertices coincide is called a \emph{double vertex}. Two sides ${L_{ij}}$ and ${L_{i'j'}}$ are {\it opposite} each other if $\{i,j\} \cap \{i',j'\} = \varnothing$; otherwise, the two sides are {\it adjacent}. The vertex configuration is {\it degenerate} if some three vertices (with coincidences allowed) are collinear. If at least three of the sides of~$Q$ extend the sides of a {nondegenerate} triangle, then the vertex configuration $Q$ is a {\it quadrangle}.  \end{definition}

A quadrangle is a vertex configuration, but the converse need not be true since a vertex configuration could  for example  consist of four coincident points and six lines through them, which is not a quadrangle according to our definition. A version of the following proposition also appears in \cite[Proposition~2.2]{OWEnt} in a broader algebraic treatment of quadrangles. We include a  proof here to keep the present paper   self-contained. The three cases of the proposition are illustrated in Figure~\ref{degenerate quadrangles figure}.

      \begin{figure}[ht]
     \begin{center}
\includegraphics[width=.95\textwidth,scale=.03]{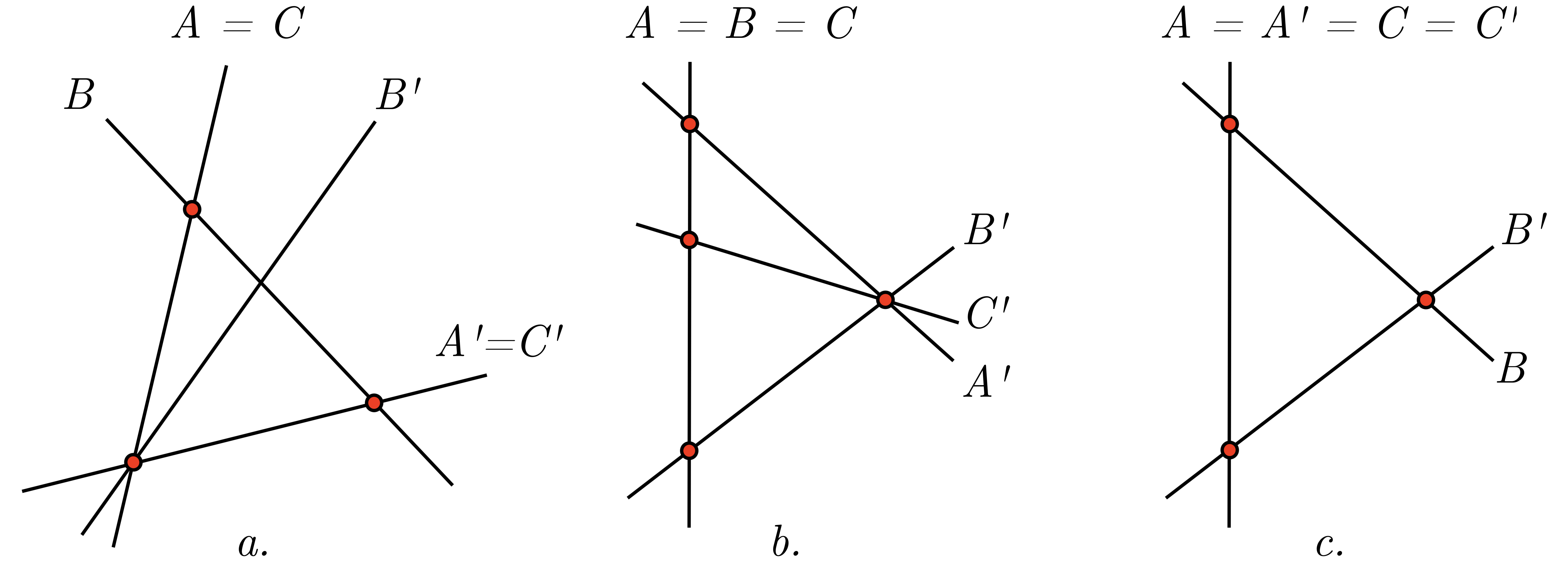}
 \end{center}
\caption{The three types of degenerate quadrangles. (a) A double-side quadrangle: two pairs, when taken as sets,  are equal, and there is one double vertex. (b) A triple-side quadrangle: three {pairwise} adjacent sides coincide in a triple side. (c) A quadruple-side quadrangle: two pairs of opposite sides collapse to a single shared side, while the third pair consists of two distinct {nonparallel} sides. }
  \label{degenerate quadrangles figure}
\end{figure}
\begin{proposition}\label{degenerate classification}
Each degenerate quadrangle satisfies  exactly one of the following statements.
\begin{enumerate}
\item[\rm (a)]   {Two} pairs of opposite sides have the same underlying two-element set of lines, and neither of those lines belongs to the third pair.
\item[\rm (b)] Three  pairwise adjacent sides coincide.
\item[\rm (c)] All four sides occurring in two pairs of opposite sides coincide in one line, while the third pair consists of two distinct nonparallel lines different from that line.
\end{enumerate}
\end{proposition}

\begin{proof}
 If all four vertices are collinear, the definition of a quadrangle   excludes four distinct vertices, three or more coincident vertices, and one double vertex with two other distinct vertices, since the sides then, respectively, coincide, are concurrent, or have at most two directions. Thus, after relabeling, $v_1=v_2=p$ and $v_3=v_4=q$. Then $L_{13}=L_{14}=L_{23}=L_{24}$
  is the line through $p$ and $q$. By the definition of a quadrangle, the remaining sides $L_{12}$ and $L_{34}$ are nonparallel and different from that line, and so case~(c) applies.

 If the vertices are not all collinear, degeneracy leaves either one double vertex or four distinct vertices with three collinear. With a double vertex, we can relabel so that $v_1=v_2=p$ and $p,v_3,v_4$ are not collinear. Then $L_{13}=L_{23}$ and $L_{14}=L_{24}$, so two pairs of opposite sides have the same underlying set. Moreover, $L_{34}$ lies outside this set and $L_{12}\ne L_{34}$. If $L_{12}$ belongs to the common set, three pairwise adjacent sides coincide, {as in case~(b)}; otherwise case~(a) holds. In the second case, {we may relabel so that} $L_{12}=L_{13}=L_{23}$, which puts us in case~(b). The greatest side multiplicities in cases~(a), (b), and~(c) are $2,3,4$, respectively, so the cases are mutually exclusive.
\end{proof}

For a line occurring among the sides of a quadrangle, its \emph{side multiplicity} is the number of its occurrences among the six sides. We say a side of multiplicity two, three, or four   a \emph{double side}, \emph{triple side}, or \emph{quadruple side}, respectively. A degenerate quadrangle is called \emph{double-side}, \emph{triple-side}, or \emph{quadruple-side} according to whether its greatest side multiplicity is two, three, or four. Thus cases~(a), (b), and~(c) of Proposition~\ref{degenerate classification} are double-side, triple-side, and quadruple-side, respectively. In a triple-side quadrangle, three {pairwise} adjacent sides coincide in its triple side; see Figure~\ref{degenerate quadrangles figure}(b).

\section{The Nash manifold of quadrangles}

 For a vertex configuration $Q$, let $v_1,\ldots,v_4\in\R^2$ be its four vertices, and write $v_i=(x_i,y_i)$. Write  its normals as $[t_{ij}:u_{ij}]$. Recalling that our vertex configurations are labeled, we can identify these configurations with the points $Q=(v_i;[t_{ij}:u_{ij}])$ in $(\R^2)^4\times(\P^1)^6$ satisfying
$$t_{ij}(x_i-x_j)+u_{ij}(y_i-y_j)=0 \quad\text{for all }1\leq i<j\leq4.$$
For each $i<j$, the corresponding side is. $$L_{ij}:t_{ij}(X-x_i)+u_{ij}(Y-y_i)=0.$$

For the rest of the article, we denote by $\Q$ the set of vertex configurations that are quadrangles. The set $\Q$ inherits a metric from the product metric on  $(\R^2)^4 \times (\P^1)^6$, where $\R^2$ has the Euclidean metric and  $\P^1$ has the standard round  metric, given for  points $[u],[v] \in {\P^1}$ by $$d_{\P^1}([u],[v])=\arccos\left(\frac{|\langle u,v\rangle|}{\|u\|\|v\|}\right).$$

The group ${\operatorname{Aff}(\R^2)}$ of affine transformations of the plane acts on $\Q$ by acting on each component of $(\R^2)^4 \times (\P^1)^6$ in the obvious way. Namely, if $g \in {\operatorname{Aff}(\R^2)}$ and $Q = (v_i;[t_{ij}:u_{ij}]) \in \Q,$ 
then $g \cdot v_i = g(v_i)$, and $g $ acts on each $ [t_{ij}:u_{ij}]$ by sending $[t_{ij}:u_{ij}]$ to the normal   of the image under $g$ of any line with normal $[t_{ij}:u_{ij}]$.  Equivalently, if ${S}$ is the linear part of $g$, the induced action on $\P^1$ sends $[t:u]$ to the projective class of ${S^{-T}}(t,u)^T$. It follows that the sides and vertices of $Q$ are sent by $g$ to the sides and vertices of $g \cdot Q$, so that the action on the points in $\Q$ corresponds to the action on the configuration in the plane that $Q$ represents.

We show in this section that the space of quadrangles  $\Q$ is a Nash manifold, and that the property of being a Nash manifold descends to $\Q/G$ if $G$ is a  {Nash subgroup} of the Nash group ${\operatorname{Sim}(\R^2)}$ of similarities of the plane. We recall first the relevant definitions.

A {\it Nash manifold of dimension $n$} is a real analytic manifold $M$ together with a finite atlas $\{(U_i,\varphi_i)\}$ such that each chart map $\varphi_i : U_i \to V_i$ is a homeomorphism onto an open semialgebraic subset $V_i\subseteq\mathbb{R}^n$, and all transition maps
 $$\varphi_j\circ\varphi_i^{-1} : \varphi_i(U_i\cap U_j) \longrightarrow \varphi_j(U_i\cap U_j)$$
are real analytic with semialgebraic graphs. A map $f : M \to N$ between Nash manifolds $M$ and $N$ is a {\it Nash map} if {it is real analytic and its graph is semialgebraic in $M\times N$, meaning that its intersection with each product chart from the finite Nash atlases has a semialgebraic coordinate image; see \cite[Lemma~2.11]{Sun}.}

A {\it Nash group} is a Lie group that is a Nash manifold and whose multiplication and inversion maps are Nash. A {\it Nash subgroup} of ${\operatorname{Aff}(\R^2)}$ is a semialgebraic subgroup of ${\operatorname{Aff}(\R^2)}$, where ${\operatorname{Aff}(\R^2)}$ is identified with ${\rm GL}_2(\R) \times \R^2$. By \cite[Proposition~3.2]{Sun}, every such subgroup is a closed Nash submanifold, and hence a Lie subgroup, of ${\operatorname{Aff}(\R^2)}$. The Nash subgroups include the algebraic subgroups of ${\operatorname{Aff}(\R^2)}$. In particular, they include the orientation-preserving and full isometry groups and the orientation-preserving and full similarity groups.

 The upshot of working in the Nash category rather than the category of smooth manifolds and smooth maps is that our methods are {\it almost} purely algebraic. As Artin and Mazur \cite[p.~85]{AM} put it, ``The Nash category is extremely useful in that it occupies a hybrid position: it is `dense' in differential topology, and it is amenable to algebraic geometric techniques.'' The main obstacle to pushing our approach all the way into the category of real varieties is that at several crucial steps we use the positive square-root function on $(0,\infty)$, which is Nash but not regular, and compositions of it with strictly positive Nash functions. Thus the Nash category allows for an approach that is largely algebraic while accommodating the functions needed in our constructions.

\begin{lemma} \label{Q is Nash}
The space  $\Q$ of quadrangles is an $8$-dimensional Nash manifold.
\end{lemma}

\begin{proof}
Let  ${\mathcal W} \subseteq(\R^2)^4\times(\P^1)^6$ be the set of vertex configurations  for which no three vertices coincide. If three vertices coincided, every side would pass through their common point, so no three sides could extend the sides of a nondegenerate triangle. Hence every quadrangle belongs to ${\mathcal W}$. Fix $Q=(v_i;[t_{ij}:u_{ij}])\in{\mathcal W}.$ For each $i<j$, let $U_{ij}$ be one of the two standard affine charts of $\P^1$ containing $[t_{ij}:u_{ij}]$, and let ${\mathcal U}=(\R^2)^4\times\prod_{i<j}U_{ij}$. Then $\mathcal U$ is an open   neighborhood of $Q$ in $(\R^2)^4\times(\P^1)^6$. If $U_{ij}$ is the chart $t_{ij}\ne0$, set $s_{ij}=u_{ij}/t_{ij}$ and define
\begin{equation}\label{f t chart}
f_{ij}=(x_i-x_j)+s_{ij}(y_i-y_j).
\end{equation}
If $U_{ij}$ is the chart $u_{ij}\ne0$, set $s_{ij}=t_{ij}/u_{ij}$ and define
\begin{equation}\label{f u chart}
f_{ij}=s_{ij}(x_i-x_j)+(y_i-y_j). \end{equation}
On the overlap of the two charts, the second function is $t_{ij}/u_{ij}$ times the first, so the two functions have the same zero set. Thus the set of vertex configurations lying in ${\mathcal U}$ is the common zero set of these six functions $f_{ij}$. The fourteen local coordinates on ${\mathcal U}$ are the four $(x_i,y_i)$ and the six $s_{ij}$, and via  these coordinates, ${\mathcal U}$ is identified with $\R^{14}$, and each $f_{ij}$ is polynomial in the local coordinates and so is a  Nash function. We claim 
that their six gradients at $Q$ are linearly independent.

Suppose there are $a_{ij} \in \R$ such that 
\begin{equation}\label{gradient relation}
\sum_{i<j}a_{ij}\,\nabla f_{ij}(Q)=0.
\end{equation}
We show that all of the $a_{ij}$ are zero by considering separately the cases $v_i\neq v_j$ and $v_i=v_j$. Suppose first that  $v_i\neq v_j$, and consider the $s_{ij}$-coordinate of this vector equality. For $\{k,\ell\}\neq\{i,j\}$, the function $f_{k\ell}$ is independent of $s_{ij}$, so the $s_{ij}$-coordinate of $\nabla f_{k\ell}(Q)$ is zero. Thus the $s_{ij}$-coordinate of the vector equality reduces to $$a_{ij}\frac{\partial f_{ij}}{\partial s_{ij}}(Q)=0.$$
From the equations~(\ref{f t chart}) and~(\ref{f u chart}),
$$\frac{\partial f_{ij}}{\partial s_{ij}}(Q) = \begin{cases} y_i-y_j,&t_{ij}\ne0,\\ x_i-x_j,&u_{ij}\ne0. \end{cases}$$
This partial derivative is nonzero since if it vanished in the first chart, then $y_i-y_j=0$, and the equation $f_{ij}(Q)=0$ would imply $x_i-x_j=0$. If it vanished in the second chart, the same equation would similarly imply $y_i-y_j=0$. Either conclusion would contradict $v_i\ne v_j$, and so we conclude that $a_{ij}=0$ when $v_i\ne v_j$.

Since the coefficients for pairs $\{i,j\}$ with $v_i\ne v_j$ are already zero, relation~(\ref{gradient relation}) now involves only the gradients associated to double vertices. Suppose $v_i = v_j$ for some  $i<j$, so that their common point is a double vertex. For the associated index pair, the six $s_{k\ell}$-coordinates are all zero since $f_{ij}$ is independent of $s_{k\ell}$ when $\{k,\ell\}\neq\{i,j\}$, while equations~(\ref{f t chart}) and~(\ref{f u chart}), together with the assumption $v_i=v_j$, imply $$\frac{\partial f_{ij}}{\partial s_{ij}}(Q)=0.$$
The only nonzero blocks of $\nabla f_{ij}(Q)$ occur in the $v_i$- and $v_j$-coordinates. These blocks are, respectively, $\pm(1,s_{ij})$ in the chart $t_{ij}\ne0$, and $\pm(s_{ij},1)$ in the chart $u_{ij}\ne0$. Two distinct double vertices involve disjoint pairs of indices, since otherwise three vertices would coincide. Consequently, the nonzero coordinate blocks of the gradient for $\{i,j\}$ do not occur in the gradient associated to any other double vertex. Inspection of either of these two blocks in a linear combination of the gradients shows that the coefficient of $\nabla f_{ij}(Q)$ must be zero, since each displayed block is nonzero. Hence these gradients are linearly independent. Therefore, $a_{ij}=0$ for every pair $\{i,j\}$ defining a double vertex, which shows that $a_{ij}=0$ in all cases.

The condition that no three vertices coincide is open and semialgebraic. Hence, for each choice of the affine charts $U_{ij}$, the intersection ${\mathcal W}\cap{\mathcal U}$ is an open semialgebraic subset of $F^{-1}(0)$, where $F=(f_{ij})_{i<j}$. Since the preceding argument shows that the six gradients, equivalently the six rows of the Jacobian of $F$, are linearly independent at every point of ${\mathcal W}$, the Jacobian has rank $6$ there. Thus the Nash implicit-function theorem \cite[Corollary~2.9.8, p.~57]{BCR}  shows that ${\mathcal W}$ is a Nash submanifold of the Nash manifold $(\R^2)^4\times(\P^1)^6$ of dimension $14-6=8.$ Finally, for each fixed triple of sides, the conditions that their normals be pairwise distinct and that the three lines not be concurrent are open semialgebraic conditions. Since $\Q$ is the finite union of these open sets over all triples of sides, $\Q$ is an open semialgebraic subset of ${\mathcal W}$. Therefore, $\Q$ is an $8$-dimensional Nash manifold.
\end{proof}

The fact that by Definition~\ref{quadrangle def} a quadrangle must contain a {nondegenerate} triangle is crucial in showing that the group ${\operatorname{Sim}(\R^2)}$ of similarity transformations  acts freely and properly on~$\Q$.

\begin{lemma} \label{free proper} Every  closed subgroup of ${\operatorname{Sim}(\R^2)}$ acts freely and properly on~$\Q$.
\end{lemma}

\begin{proof} Let $G$ be a closed subgroup of ${\operatorname{Sim}(\R^2)}$.  To see that $G$ acts freely on $\Quad$, let $g\in G$ and suppose $g\cdot Q=Q$ for some $Q\in\Quad$. Since $Q$ is a quadrangle, it has three sides that form the sides of a {nondegenerate} triangle. By the definition of the action, $g\cdot Q=Q$ implies that $g$ preserves each of these three labeled sides as sets, so it fixes their three pairwise intersections. Since these points are noncollinear, and the identity is the only affine transformation fixing three noncollinear points, $g$ is the identity, so $G$ acts freely on~$\Quad$.

 To prove properness, it suffices by \cite[Proposition~21.5, p.~543]{Lee} to consider sequences $Q_n\in\Q$ and $g_n\in G$ such that $Q_n\to Q$ and $g_n\cdot Q_n\to Q'$ in $\Q$, and to show that $(g_n)$ has a convergent subsequence. For $R=(v_i;[t_{ij}:u_{ij}])\in\Q$, set
$$d(R)=\max_{1\le i<j\le4}\|v_i-v_j\|.$$
The function $d$ is continuous and positive on $\Q$ since if $d(R)=0$, then all four vertices coincide and all six sides are concurrent, contrary to the definition of a quadrangle. Write $g_n(v)=\lambda_nB_nv+b_n,$ where $\lambda_n>0$, $B_n\in O(2)$, and $b_n\in\R^2$. Since $g_n$ scales distances by $\lambda_n$,
$$\lambda_n=\frac{d(g_n\cdot Q_n)}{d(Q_n)} \longrightarrow\frac{d(Q')}{d(Q)}>0.$$
{By compactness of $O(2)$, after passing to a subsequence we may assume that} $B_n\to B$. Let $v_{1,n}$ and $v'_{1,n}=g_n(v_{1,n})$ be the first vertices of $Q_n$ and $g_n\cdot Q_n$, respectively. Both vertex sequences converge, and so $b_n=v'_{1,n}-\lambda_nB_nv_{1,n}$ converges. Thus $(g_n)$ has a convergent subsequence in ${\operatorname{Sim}(\R^2)}$. Because $G$ is closed, the limit lies in $G$, which proves the lemma. \end{proof}

\begin{example} The same argument as in Lemma~\ref{free proper} shows that the action of the group ${\operatorname{Aff}(\R^2)}$ of affine transformations on $\Q$ is free. However, the action is not proper. To see this, let $Q_n$ be the quadrangle with vertices $$v_1=(0,0),\:\: v_2=(1,0),\:\:   v_3=(0,1),\:\: v_4=\left(1+\frac{1}{n},1\right).$$ Then $Q_n\in \Q$ for all $n$, and $Q_n \to Q$ where $Q$ is the rectangle with vertices $$ (0,0),\ (1,0),\ (0,1),\ (1,1).$$ Let $g_n$ be the affine transformation of $\R^2$ defined by  $g_n(x,y)=(x,y/n).$ The sequence ${g_n}$ leaves every compact subset of  ${\operatorname{Aff}(\R^2)}$ since the matrices $\operatorname{diag}(1,1/n)$ leave every compact subset of $\mathrm{GL}_2(\mathbb{R})$.

{The images under $g_n$ of the vertices of $Q_n$ are} $$g_n(v_1)=(0,0),\:\: g_n(v_2)=(1,0),\:\: g_n(v_3)=\left(0,\frac{1}{n}\right),\:\: g_n(v_4)=\left(1+\frac{1}{n},\frac{1}{n}\right).$$ The transformed sides are
$$\begin{array}{lll} g_n(L_{12}):y=0, & g_n(L_{13}):x=0, & g_n(L_{14}):y=x/(n+1),\\ g_n(L_{23}):y=(1-x)/n, & g_n(L_{24}):y=x-1, & g_n(L_{34}):y=1/n. \end{array}$$
   Thus $g_n(L_{12}),g_n(L_{14}),g_n(L_{23})$, and $g_n(L_{34})$ converge to $y=0$, while $g_n(L_{13})$ and $g_n(L_{24})$ converge to $x=0$ and $y=x-1$, respectively. Hence $g_n\cdot Q_n$ converges in $(\R^2)^4\times(\P^1)^6$ to a vertex configuration $Q'$. Its first and third vertices are $(0,0)$, its second and fourth vertices are $(1,0)$, and its six sides are the corresponding limits above. The three sides $L_{12},L_{13}$, and $L_{24}$ of $Q'$ have distinct normals and are not concurrent, so $Q'\in\Q$. Let $$K= \{Q_n:n\ge1\}\cup\{g_n \cdot Q_n:n\ge1\} \cup \{Q,Q'\}.$$ Then $K$ is compact and $g_nK\cap K\neq\varnothing$ for every $n$ since $g_n \cdot Q_n$ is in this intersection. However, since the sequence $\{g_n\}$ has no convergent subsequence in ${\operatorname{Aff}(\R^2)}$, it follows that $$G_K=\{g\in{\operatorname{Aff}(\R^2)}:gK\cap K\neq\varnothing\}$$ is not compact, and so the action of ${\operatorname{Aff}(\R^2)}$ on $\Q$ is not proper.
\end{example}

The following lemma, due to Sun, will be needed in Lemma~\ref{Nash lemma} and later results.

\begin{lemma} \cite[Lemma~2.7]{Sun} \label{Nash descent} Let $\pi:M\to B$ be a surjective Nash submersion, and let $f:B\to N$ be a map to a Nash manifold. Then $f$ is Nash if and only if $f\circ\pi$ is Nash. Consequently, if $F:M\to N$ is a Nash map that is constant on the fibers of $\pi$, then the induced map $\overline{F}:B\to N$ is Nash.
\end{lemma}

The next lemma is implicit in a theorem of Aizenbud, Carmeli, Gourevitch and Radzivilovski in  \cite{ACGR}, but since the version we need is not explicitly stated, we indicate how to derive what we need from the cited result.

\begin{lemma} \label{Nash lemma} {\em (cf.~\cite[Theorem~3.3.17]{ACGR})} If $G$ is a Nash group that acts Nashly, freely and properly on a Nash manifold $X$, then $X/G$ is a Nash manifold of dimension $\dim X-\dim G$ and $X/G$ has a unique Nash structure such that the quotient map $\pi: X \longrightarrow X/G$ is a submersion of Nash manifolds and   a principal $G$-bundle.
\end{lemma}

\begin{proof}
  To prove the lemma,  it is enough by  a theorem of Aizenbud, Carmeli, Gourevitch and Radzivilovski \cite[Theorem~3.3.17]{ACGR} to show that the image $R$ of the map $$\nu:G \times X \longrightarrow X \times X:(g,x) \longmapsto (x,g \cdot x)$$ is embedded, closed and semialgebraic, and the projection map of $R$ onto the first component of $X \times X$ is a Nash map that is a submersion of smooth manifolds.

Since Nash maps have semialgebraic graphs, the map $\nu$ is semialgebraic, and hence $R=\nu(G\times X)$ is semialgebraic by the Tarski-Seidenberg Theorem \cite[Proposition~2.2.7, p.~29]{BCR}. By assumption, $G$ acts freely and properly on $X$.  Properness of the action means that the map $$G\times X \longrightarrow X\times X: (g,x)\longmapsto (g\cdot x,x)$$
is proper, and hence, after swapping the two factors of $X\times X$, the map $\nu$ is proper. Since $X\times X$ is locally compact, the image $R$ of this proper map is closed.

By the smooth quotient manifold theorem \cite[Theorem 21.10, pp.~544--545]{Lee}, $X/G$ is a smooth manifold, the quotient map $\pi:X\to X/G$ is a smooth submersion, and $\dim(X/G)=\dim X-\dim G.$ The orbit relation is the fiber product
$$R=\{(x,y)\in X\times X:\pi(x)=\pi(y)\}=X\times_{X/G}X.$$ Let $(x,y)\in R$ and set $q=\pi(x)=\pi(y)$. By the submersion theorem, after shrinking around $q$, there are charts about $x$ and $y$ whose coordinate images are $U\times V_1$ and $U\times V_2$, respectively, in which $\pi$ is projection onto $U$. Thus points in the two charts have coordinates $(u,a)$ and $(v,b)$, respectively, and $\pi(u,a)=u$ and $ \pi(v,b)=v.$ In the resulting product chart, $R$ is the set
$$\{((u,a),(u,b)):u\in U,\ a\in V_1,\ b\in V_2\}.$$ This is an embedded submanifold because it is the image of the smooth retraction $$((u,a),(v,b))\longmapsto((u,a),(u,b)).$$
Since $(x,y)$ was arbitrary, $R$ is a smooth embedded submanifold of $X\times X$. In these coordinates, the first projection $R\to X$ is $(u,a,b)\mapsto(u,a),$ and hence is a smooth submersion.

 Since $R$ is closed and semialgebraic and is a smooth embedded submanifold of the Nash manifold $X\times X$, \cite[Proposition~2.9]{Sun} shows that it carries the induced structure of an embedded Nash submanifold. The first projection $R\to X$ is Nash, being the restriction of the Nash projection $X\times X\to X$, and the preceding paragraph shows that it is a smooth submersion. Thus it is a Nash submersion. This verifies all the hypotheses of \cite[Theorem~3.3.17]{ACGR}, so $X/G$ is a Nash manifold and $\pi:X\to X/G$ is a Nash submersion.

It remains to prove the principal-bundle assertion and uniqueness. It is a   consequence of the quotient manifold theorem 
\cite[Problem~21-6, p.~560]{Lee} that $\pi$ is a smooth principal $G$-bundle. Since $\pi$ is a Nash submersion, local sections of~$\pi$ may be chosen Nash \cite[Theorem~A.0.1]{AG}. For such a section $\sigma:U\to X$, the corresponding smooth local trivialization
$$G\times U\longrightarrow\pi^{-1}(U):(g,y)\longmapsto g\cdot\sigma(y)$$
is Nash because the action and $\sigma$ are Nash. Its inverse is smooth, and it is semialgebraic because its graph is obtained from the graph of the original map by interchanging the two factors. Hence its inverse is Nash, and $\pi$ is a Nash principal $G$-bundle.

Finally, suppose that $X/G$ carries two Nash structures for which the quotient map is a Nash submersion, and write $\pi_i:X\to(X/G)_i$ for these two Nash quotient maps. The identity map $i:(X/G)_1\to(X/G)_2$ satisfies $i\circ\pi_1=\pi_2$. Since $\pi_1$ is a surjective Nash submersion and $\pi_2$ is Nash, Lemma~\ref{Nash descent} shows that $i$ is Nash. Reversing the roles of the two structures shows that $i^{-1}$ is Nash. Thus the two Nash structures coincide, which proves uniqueness and completes the proof.
 \end{proof}

\begin{theorem} \label{main Nash}
Let $G$ be a Nash subgroup of ${\operatorname{Sim}(\R^2)}$. The orbit space $\Quad/G$ is a Nash manifold of dimension $8-\dim G$ and has a unique Nash structure such that the quotient map $\pi: \Quad \rightarrow \Quad/G$ is a submersion of Nash manifolds and   a principal $G$-bundle.

\end{theorem}

\begin{proof}  By Lemma~\ref{Q is Nash}, $\Q$ is a Nash manifold of dimension $8$. Let $G$ be a Nash subgroup of ${\operatorname{Sim}(\R^2)}$. By Lemma~\ref{free proper}, $G$ acts freely and properly on $\Quad$. The action of ${\operatorname{Aff}(\R^2)}$ on $(\mathbb{R}^2)^4 \times (\mathbb{RP}^1)^6$ is Nash and preserves $\Q$. Hence its restriction to every Nash subgroup $G\subseteq{\operatorname{Sim}(\R^2)}$ is a Nash action on $\Q$. The theorem now follows from Lemma~\ref{Nash lemma}.
 \end{proof}

\begin{remark} \label{metric remark}
Let $G$ be a Nash subgroup of the group of isometries of the plane, and let $\rho$ denote the restriction to $\Q$ of the product     distance on $(\mathbb{R}^2)^4 \times (\mathbb{RP}^1)^6$ determined by the Euclidean metric on $\mathbb{R}^2$ and the round quotient metric on $\mathbb{RP}^1$. Since $G$ acts properly on $\Q$ by Lemma~\ref{free proper} and preserves $\rho$, 
     \cite[Theorem~4.3.4, pp.~318--319]{Palais} shows that the orbit space $\Q/G$ has a metric inducing the quotient topology, given by
$$\overline{\rho}([Q_1],[Q_2]) =\inf\{\rho(Q_1',Q_2'):Q_1'\in[Q_1],\ Q_2'\in[Q_2]\}.$$
\end{remark}

\smallskip

\section{The form of a quadrangle}\label{form section}

In this section we associate to a quadrangle an invariant that we call its {quadrangle }form. The form is unchanged under translations and uniform scalings, and in the next section we will use it to define a similarity invariant, aspect, for quadrangles. To define our invariant, we first establish some terminology.

If $H\in\R[X,Y]$ has degree~$2$, let $C_H$ denote the affine conic represented by the equation class v$[H]=\{\lambda H:\lambda\in\R^\times\}.$ We call $H$ a {\it defining polynomial} of $C_H$. Its  real point set is $C_H(\R)=V(H)=\{p\in\R^2:H(p)=0\}.$ The projective realization of $C_H$ is defined by the degree-two homogenization of $H$. We call $C_H$ a {\it hyperbola} if its projective realization meets the line at infinity in two distinct real points.  A point  $p$ is a center of $C_H$ if $\nabla H(p)=0$.

Pairs of opposite sides of a quadrangle can be viewed as conics in the following way. If $\mathsf A=\{A,A'\}$ is a pair of lines, a {\it line-pair polynomial} for $\mathsf A$ is a product $F_{\mathsf A}=\ell_A\ell_{A'}$, where $\ell_A$ and $\ell_{A'}$ are nonzero 
linear polynomials whose zero sets are $A$ and $A'$. If the two sides coincide, we use the same polynomial twice. A line-pair polynomial is uniquely determined up to multiplication by a nonzero scalar  and thus determines a conic. Two pairs of opposite sides are said to be distinct if they are distinct as unordered pairs of lines.

The quadrangle form, which we define next, depends on the what we call the center locus of  the quadrangle. The reason for this name is explained by Proposition~\ref{form coefficients} and illustrated by Figure~\ref{center locus figure}.

\begin{definition} \label{form def} For any two polynomials $F,G\in\R[X,Y]$, set
$$J(F,G)=(\partial_XF)(\partial_YG)-(\partial_YF)(\partial_XG).$$
Let $Q$ be a quadrangle, and let $F,G$ be line-pair polynomials for two distinct pairs of opposite sides of $Q$.  We call the conic defined by $J(F,G)=0$  the {\it {center-}locus conic} $C_Q$ of $Q$. Thus $C_Q=C_{J(F,G)}$. Its real point set $C_Q(\R)=V(J(F,G))$ is called the {\it {center }locus} of $Q$.  Write the homogeneous quadratic part of the Jacobian  as $$J(F,G)_2=2aX^2-4bXY+2cY^2.$$ The {\it {quadrangle }form} of $Q$ is
$$\operatorname{form}(Q)=[a:b:c]\in\P^2.$$
 To the representative $(a,b,c)$ of the form, associate the symmetric matrix
\begin{equation}\label{matrixx}
M(a,b,c)=
\begin{pmatrix}
2a&-2b\\
-2b&2c
\end{pmatrix}.
\end{equation}

\end{definition}

Proposition~\ref{form coefficients} will show that $J(F,G)$ has degree $2$ and that the {center-}locus conic, {center }locus, and {quadrangle }form are independent of the choices of pairs of opposite sides and line-pair polynomials. Note also that
\begin{equation}\label{jacobian quadratic part}
J(F,G)_2=(X,Y)M(a,b,c)(X,Y)^T.
\end{equation}

We can interpret the {center }locus geometrically. Let $Q$ be a quadrangle, and let $F,G$ be line-pair polynomials for two distinct pairs of opposite sides of $Q$. The {\it pencil} of $Q$ is the set of conics $C_{\lambda F+\mu G}$, where $[\lambda:\mu] \in \P^1$. When $Q$ is {nondegenerate}, {a conic belongs to this pencil if and only if it contains the four vertices of $Q$}. The next proposition identifies the {center }locus with the set of centers of the members of the pencil of $Q$. 

\begin{figure}[htbp]
\centering
\includegraphics[width=.75\textwidth]{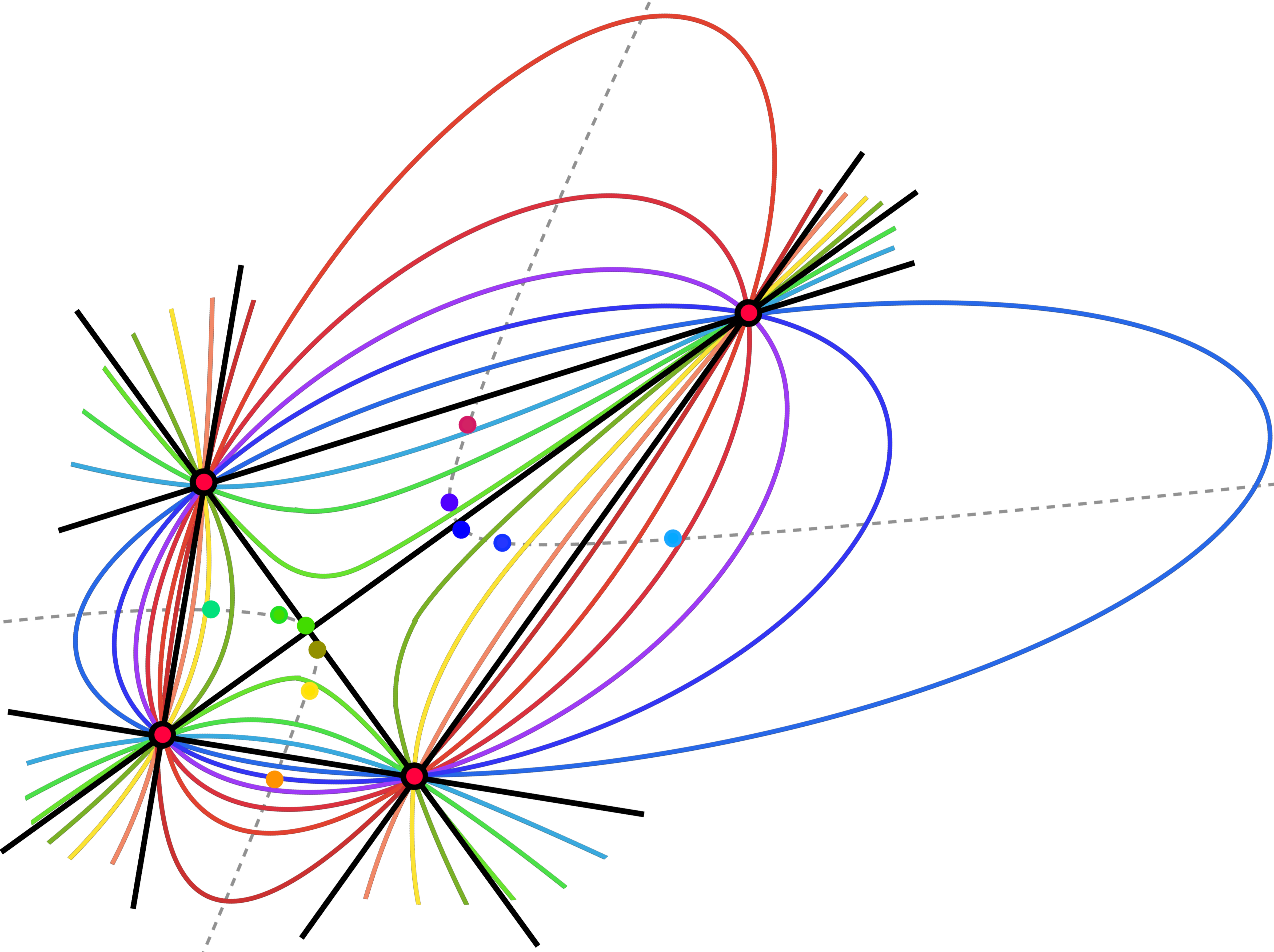}
\caption{Conics in the pencil of a   quadrangle, whose sides are shown in black. The colored points mark the centers of the corresponding conics and lie on the dashed gray center locus.}
\label{center locus figure}
\end{figure}

As follows from the definition and Proposition~\ref{form coefficients}, for a {nondegenerate} quadrangle, the projective realization of the {center-}locus conic is   the classical {\it nine-point conic} of the quadrangle \cite[Section~V]{Mint}.  It is so called because it passes through the six side midpoints of the quadrangle and the three points of intersection of opposite sides, some of which may be at infinity. ({The midpoint of a side $L_{ij}$ means the midpoint of its two vertices $v_i$ and $v_j$.}) For historical background on this conic, see \cite{Del,Vac}. In the terminology of \cite[Theorem~{7.2}]{OWEnt}, where the pencil of a quadrangle is studied in more detail, this {center }locus is also the set of {finite} bisection points of the  bisectors of~$Q$; see also \cite{OWQuad}.

The {center-}locus calculation in the next proposition is classical for a nondegenerate quadrangle \cite[No.~8, p.~171]{Casey}, but we give an argument since our quadrangles can be degenerate.

\begin{proposition}\label{form coefficients}
The pencil, {center-}locus conic, {center }locus, and {quadrangle }form of a quadrangle $Q$ are well-defined. The {center }locus {is the set of points} in $\R^2$ that are centers of members of the pencil of $Q$.
\end{proposition}

\begin{proof}
Let $F,G$ be line-pair polynomials for two distinct pairs of opposite sides of $Q$. We first show that any two such polynomials form a basis of the same two-dimensional vector space. If $Q$ is nondegenerate, its three pairs of opposite sides are distinct degenerate conics through the four vertices. By Theorem~5.13 of  \cite{BambergPenttila} and the paragraph following it, the conics through these vertices form a pencil. Hence any two of the three pairs of opposite sides generate the two-dimensional vector space underlying this pencil. If $Q$ is degenerate, we can apply Proposition~\ref{degenerate classification}. In cases~(a) and~(c), there are only two distinct pairs of opposite sides.  In case~(b), after fixing an equation $\ell_T$ for the triple side $T$, the polynomials of the distinct pairs of opposite sides have the form $\ell_Tg_i$, where the $g_i$ are linear equations vanishing at the vertex off $T$. Distinct pairs {correspond to} linearly independent $g_i$, and any two such forms {form} a basis of the two-dimensional space of linear polynomials vanishing at that vertex. It follows that the pencil is independent of the chosen pairs and line-pair polynomials.

Let $F_2,G_2$ be the homogeneous degree-two parts of $F,G$. These polynomials  are not scalar multiples of each other, since otherwise the two pairs of opposite sides would have the same two normals, with multiplicity, so that matching parallel sides from the two pairs would be adjacent and hence coincide, contrary to the pairs being distinct. Consequently, every nonzero element of $\operatorname{span}\{F,G\}$ has degree~$2$. Write $F_2=a_1X^2+a_2XY+a_3Y^2$ and $G_2=b_1X^2+b_2XY+b_3Y^2.$ Then
$$J(F,G)_2 =2(a_1b_2-a_2b_1)X^2+4(a_1b_3-a_3b_1)XY +2(a_2b_3-a_3b_2)Y^2.$$
The three expressions in parentheses are the $2\times2$ minors of
$$\begin{pmatrix} a_1&a_2&a_3\\ b_1&b_2&b_3 \end{pmatrix}.$$
Since $F_2$ and $G_2$ are not scalar multiples, this matrix has rank~$2$, so at least one of these minors is nonzero. Therefore, $J(F,G)_2\ne0$, and hence $J(F,G)$ has degree~$2$.

If $F',G'$ arise from another choice of pairs of opposite sides, then $F'=rF+sG$ and $G'=tF+uG,$ where $ru-st\ne0$. {By the alternating bilinearity of the Jacobian determinant,} $J(F',G')=(ru-st)J(F,G).$ Thus the {center-}locus conic, {center }locus, and {quadrangle }form are independent of all choices. Finally, a point $p\in\R^2$ is the center of a member of the pencil exactly when $\lambda\nabla F(p)+\mu\nabla G(p)=0$ for some nonzero pair $(\lambda,\mu)$. This is equivalent to $J(F,G)(p)=0$, which proves the assertion about the {center }locus.
\end{proof}

The {quadrangle }form defines a map from the Nash manifold $\Quad$ to $\P^2$. In fact, as we see next, it is a submersion, a property that is crucial for the analysis of the manifold $\Q$ in the next sections because it leads ultimately to a decomposition of $\Quad$ into Nash level submanifolds based on aspect.

We recall the local-section criterion for Nash submersions \cite[Theorem~A.0.1 and Proposition~A.0.4]{AG}, which asserts  that if $f:X\to Y$ is a Nash submersion at $x_0$ and $y_0=f(x_0)$, then there is an open neighborhood $W\subseteq Y$ of $y_0$ and a Nash map $\sigma:W\to X$ such that $\sigma(y_0)=x_0$ and $f\circ\sigma=\operatorname{id}_W$. Conversely, the existence of such a local section implies that $f$ is a submersion at $x_0$, by differentiating $f\circ\sigma=\operatorname{id}_W$. For the form map, the existence of a local section shows  that every sufficiently small variation of the form of $Q$ can be realized by quadrangles depending Nashly on the form and passing through $Q$.

\begin{theorem} \label{form submersion} The form  map from $\Quad$ to $ \P^2$ is a Nash submersion.

\end{theorem}

\begin{proof}
For adjacent sides $L_{i_1j_1}$ and $L_{i_2j_2}$, let $U_{i_1j_1,i_2j_2}$ be the open subset of $\Q$ consisting of quadrangles for which three of the four sides ${L_{i_1j_1}},$ ${L_{i_1'j_1'}},$ ${L_{i_2j_2}},$ ${L_{i_2'j_2'}}$ form a {nondegenerate} triangle, where $$\{i_1',j_1'\}=\{1,2,3,4\}\setminus\{i_1,j_1\},\:\: \{i_2',j_2'\}=\{1,2,3,4\}\setminus\{i_2,j_2\}.$$
To see that these open sets cover $\Q$,  choose three sides of a quadrangle that form a {nondegenerate} triangle. If two of these three sides are opposite sides of the quadrangle, then those two sides form one pair of opposite sides, and the third side belongs to another. Thus the quadrangle belongs to one of the open sets above. Otherwise, after relabeling, the three sides are ${L_{12},L_{13},L_{23}}$. Set
$$L_1={L_{23}},\:\: L_2={L_{13}},\:\:  L_3={L_{12}},\:\: L_1'={L_{14}},\:\:  L_2'={L_{24}},\:\:  L_3'={L_{34}}.$$
 For each $i$, the line $L_i'$ passes through $v_i$, which is the vertex of the triangle opposite $L_i$, and through $v_4$. If $L_i'$ is not parallel to $L_i$, then $L_i,L_i'$, together with one of the two sides of the original triangle through the vertex opposite $L_i$ and distinct from $L_i'$, form a nondegenerate triangle contained in the union of two pairs of  opposite sides. If no union of two pairs of opposite sides contained a {nondegenerate} triangle, then each $L_i'$ would be parallel to $L_i$. But the three lines through the vertices of a {nondegenerate} triangle parallel to the opposite sides are not concurrent, whereas $L_1',L_2',L_3'$ all pass through $v_4$. This proves the specified open sets cover $\Q$.

It remains to prove the theorem on one such open set $U=U_{i_1j_1,i_2j_2}$. For this fixed open set, write $A=L_{i_1j_1}$, $A'=L_{i_1'j_1'}$, $B=L_{i_2j_2}$, and $B'=L_{i_2'j_2'}$. Let $\mathcal N\subseteq(\mathbb{RP}^1)^4$ be the open set of four-tuples with at least three distinct entries, viewed as normals $[t:u]$. By the definition of $U$, the four selected normals determine a point of $\mathcal N$. Define
$$\nu:U\longrightarrow\mathcal N: Q\longmapsto \big(n_A(Q),n_{A'}(Q),n_B(Q),n_{B'}(Q)\big),$$
where $n_L(Q)$ denotes the normal of the side $L$. We first show that $\nu$ is a Nash submersion. Fix $Q_0\in U$. It suffices to construct a Nash local right inverse of $\nu$ through $Q_0$, since differentiating $\nu\circ\sigma=\operatorname{id}$ shows that $d\nu_{Q_0}$ is surjective. Since $Q_0\in U$, three of the four sides $A,A',B,B'$ form a {nondegenerate} triangle. Among these three sides there is one side from $\{A,A'\}$ and one side from $\{B,B'\}$, and these two sides are not parallel. After possibly interchanging $A$ with $A'$ and $B$ with $B'$, and then relabeling the vertices, we may assume $A=L_{12},$ $A'=L_{34},$ $B=L_{13},$ $B'=L_{24},$ with $A'$ and $B'$ nonparallel at $Q_0$.

Choose affine coordinate neighborhoods $V_A, V_{A'}, V_B, V_{B'}\subseteq\mathbb{RP}^1$ of the four normals $n_A(Q_0),$ $n_{A'}(Q_0),$ $n_B(Q_0),$ $n_{B'}(Q_0),$ and set
$$V=V_A\times V_{A'}\times V_B\times V_{B'}.$$
On these affine neighborhoods, choose Nash maps $e_L:V_L\to\mathbb R^2\setminus\{0\}$, for $L=A,A',B,B'$, such that each $e_L(\xi)$ is perpendicular to the normal $\xi$. Denote the $i$th vertex of the  quadrangle $Q_0$ by $v_i(Q_0)$, and choose real numbers $r,s$ such that
$$v_2(Q_0)=v_1(Q_0)+r e_A(n_A(Q_0)),\:\:\: v_3(Q_0)=v_1(Q_0)+s e_B(n_B(Q_0)).$$
For $\eta=(\eta_A,\eta_{A'},\eta_B,\eta_{B'})\in V$, define
$$v_1(\eta)=v_1(Q_0),\:\:\: v_2(\eta)=v_1(Q_0)+r e_A(\eta_A),\:\:\: v_3(\eta)=v_1(Q_0)+s e_B(\eta_B).$$
Because the sides $A'$ and $B'$ of $Q_0$ are nonparallel, there is a point $v_4(\eta)$, for $\eta$ near $\nu(Q_0)$, that is the intersection of the line through $v_3(\eta)$ with normal $\eta_{A'}$ and the line through $v_2(\eta)$ with normal $\eta_{B'}$. This intersection is obtained by solving a nonsingular $2\times2$ linear system, and hence varies Nashly with $\eta$.

Set $n_{12}(\eta)=\eta_A,$ $n_{34}(\eta)=\eta_{A'},$ $n_{13}(\eta)=\eta_B,$ $n_{24}(\eta)=\eta_{B'}.$ The definitions of the vertices $v_i(\eta)$ ensure that these normals satisfy the required incidence conditions for a quadrangle. At $\eta=\nu(Q_0)$, {we have by construction} $v_4(\eta)=v_4(Q_0)$. Moreover, $v_1(Q_0)\ne v_4(Q_0)$ and $v_2(Q_0)\ne v_3(Q_0)$, since if either equality held, all four sides $A,A',B,B'$ would be concurrent and hence could not contain a {nondegenerate} triangle. At $\eta=\nu(Q_0)$, the normal directions $\eta_{A'}$ and $\eta_{B'}$ are distinct, the endpoint pairs $(v_1(\eta),v_4(\eta))$ and $(v_2(\eta),v_3(\eta))$ consist of distinct points, and three of the four selected sides form a nondegenerate triangle. These properties are open, so since $\mathcal N$ is open and contains $\nu(Q_0)$, after replacing  $V$ with a possibly smaller open set we may assume that $V\subseteq\mathcal N$ and that all these properties hold for every $\eta\in V$.

For such $\eta$, define $\sigma(\eta)$ to be the vertex configuration with vertices $v_i(\eta)$, the four prescribed side normals above, and with the two remaining side normals $n_{14}(\eta)$, the normal of the line through $v_1(\eta)$ and $v_4(\eta)$, and $n_{23}(\eta)$, the normal of the line through $v_2(\eta)$ and $v_3(\eta)$. These normals vary Nashly with $\eta$ because $v_1(\eta)\ne v_4(\eta)$ and $v_2(\eta)\ne v_3(\eta)$ for every $\eta\in V$.  Thus the preceding definitions of the vertices and side normals define a Nash map $\sigma:V\to U$. By construction, $\sigma(\nu(Q_0))=Q_0$ and $\nu\circ\sigma=\operatorname{id}_V.$ Thus $\sigma$ is a Nash local section of $\nu$ through $Q_0$. Since $Q_0$ was arbitrary, $\nu$ is a Nash submersion.

For $\eta=(\eta_A,\eta_{A'},\eta_B,\eta_{B'})\in\mathcal N$, choose homogeneous linear forms $h_A,h_{A'},h_B,h_{B'}$ whose coefficient vectors represent the corresponding entries of $\eta$. Set $F=h_Ah_{A'}$ and $G=h_Bh_{B'}$.  Define $\Phi(\eta)$ to be the projective class of $J(F,G)$, viewed as a point of $\P^2$ via the basis $X^2,-2XY,Y^2$. We claim that since at least three of the four normals are distinct, $J(F,G)\ne0$. {By the coefficient calculation in the proof of Proposition~\ref{form coefficients}, $J(F,G)=0$ would force $F$ and $G$ to be proportional.} Unique factorization then shows that the two unordered normal pairs coincide, including multiplicities{, contradicting the fact that there are at least three distinct normals}. The chosen forms are homogeneous, so $J(F,G)$ is homogeneous of degree $2$.  If each representative $h_L$ is replaced by $\rho_Lh_L$, where $\rho_L\in\R^\times$, then
$$J(F,G)\longmapsto \rho_A\rho_{A'}\rho_B\rho_{B'}J(F,G).$$
Hence the projective class of $J(F,G)$ is independent of the chosen representatives, and $\Phi$ is well-defined. 

To see that $\Phi$ is Nash, cover $\mathcal N$ by its intersections with products of the two standard affine charts of $\mathbb{RP}^1$. On each such product chart, write $\eta_L=[a_L:b_L]$ for $L=A,A',B,B'$. If the chosen chart in the $L$-component is $b_L\ne0$, set $t_L=a_L/b_L$ and take $h_L=t_LX+Y$; if it is $a_L\ne0$, set $s_L=b_L/a_L$ and take $h_L=X+s_LY$. The three coefficients of $J(h_Ah_{A'},h_Bh_{B'})$ are polynomial functions of the four affine chart coordinates and, as shown above, do not vanish simultaneously on $\mathcal N$. Thus they define a polynomial map into $\R^3\setminus\{0\}$. On each such chart, $\Phi$ is the composition of this map with the Nash projectivization map $\pi:\R^3\setminus\{0\}\longrightarrow\P^2.$ Hence $\Phi$ is Nash.

 For a quadrangle, write an   equation for each selected side $L$ as $\ell_L=h_L+d_L$, where $h_L$ is its homogeneous linear part. Then
$$J(\ell_A\ell_{A'},\ell_B\ell_{B'})_2 =J(h_Ah_{A'},h_Bh_{B'}),$$
because only the degree-one parts of the two gradients contribute to the homogeneous degree-two part of the Jacobian determinant. Therefore, on $U$, the form map is the composition $U\stackrel{\nu}{\rightarrow}\mathcal N \stackrel{\Phi}{\rightarrow}\P^2.$

It remains to show that $\Phi$ is a Nash submersion. Fix $\eta_0\in\mathcal N$. Choose homogeneous coordinates on $\mathbb{RP}^1$ for which the four entries of $\eta_0$ lie in the affine chart where the second coordinate is nonzero, and restrict to a neighborhood $W\subseteq\mathcal N$ of $\eta_0$ on which this remains true. After rescaling, write $h_L=t_LX+Y$ for $L=A,A',B,B'$. This coordinate choice amounts to an invertible linear change of $X,Y$, which induces a projective automorphism of $\P^2$, so since composition with a projective automorphism does not change differential rank, we may compute the rank of the form map in these coordinates. A direct expansion of $J(F,G)_2$ now defines a map
$$\widetilde{\Phi}:W\longrightarrow\mathbb R^3\setminus\{0\}:\eta\longmapsto \bigl(a(\eta),b(\eta),c(\eta)\bigr)$$
such that $\Phi(\eta)=[a(\eta):b(\eta):c(\eta)].$ More explicitly,
\begin{equation}\label{affine form coefficients}
\widetilde{\Phi}(\eta)=
\big(t_A t_{A'}(t_B+t_{B'})-t_Bt_{B'}(t_A+t_{A'}),
\ t_Bt_{B'}-t_At_{A'},
\ t_A+t_{A'}-t_B-t_{B'}\big).
\end{equation}
The Jacobian of $\widetilde{\Phi}$ with respect to $t_A,t_{A'},t_B,t_{B'}$ is
\begingroup
\setlength{\arraycolsep}{3.5pt}
$$\begin{pmatrix} t_{A'}(t_B{+}t_{B'})-t_B t_{B'} & t_A(t_B{+}t_{B'})-t_B t_{B'} & t_A t_{A'}-t_{B'}(t_A{+}t_{A'}) & t_A t_{A'}-t_B(t_A{+}t_{A'})\\ -t_{A'} &  -t_A & t_{B'} & t_B\\ 1 & 1 & -1 & -1 \end{pmatrix}.$$
\endgroup
The four maximal minors, obtained by deleting respectively the $t_A,t_{A'},t_B,t_{B'}$ columns, are, up to sign,
$$\begin{aligned} &(t_A-t_B)(t_A-t_{B'})(t_B-t_{B'}), &\: &(t_{A'}-t_B)(t_{A'}-t_{B'})(t_B-t_{B'}),\\ &(t_{A'}-t_A)(t_A-t_B)(t_{A'}-t_B), &\: &(t_{A'}-t_A)(t_A-t_{B'})(t_{A'}-t_{B'}). \end{aligned}$$
Since at least three of the four normals are distinct, one of these products is nonzero. Hence this Jacobian matrix has rank $3$. Since $\pi$ is a Nash submersion and $\Phi|_W=\pi\circ\widetilde{\Phi}$, the surjectivity of $d\widetilde{\Phi}_{\eta_0}$ implies that $d\Phi_{\eta_0}$ is surjective. Since $\eta_0$ was arbitrary, $\Phi$ is a Nash submersion. Therefore, the form map $\Phi\circ\nu$ is a Nash submersion on $U$. Since the open sets $U_{i_1j_1,i_2j_2}$ cover $\Q$, the form map is a Nash submersion on all of $\Q$.
\end{proof}

\begin{remark} \label{form remark}
The behavior of the form map with respect to affine transformations can be deduced from the Jacobian construction in Definition~\ref{form def}. Let $F$ and $G$ be line-pair polynomials for two distinct pairs of opposite sides of $Q$. If $g(x)=Sx+p$ is an invertible affine transformation, put $F^g=F\circ g^{-1}$ and $ G^g=G\circ g^{-1}$. {By the chain rule,}
$$J(F^g,G^g)=\det(S)^{-1}J(F,G)\circ g^{-1}.$$
Write $[M(Q)]$ for the projective class of $M(a,b,c)$, where $[a:b:c]=\operatorname{form}(Q)$. Taking homogeneous degree-two parts, translations contribute only lower-degree terms and the scalar $\det(S)^{-1}$ does not change the projective class. Therefore, by~(\ref{jacobian quadratic part}),
$$[M(g\cdot Q)]=[S^{-T}M(Q)S^{-1}].$$
Equivalently, representatives may be chosen so that
$$M(g\cdot Q)=S^{-T}M(Q)S^{-1}.$$
In particular, the form is invariant under translations and uniform scalings. If $S$ is orthogonal, representatives may be chosen so that $M(g\cdot Q)=SM(Q)S^T$.
\end{remark}

\section{Aspect and sphericity of a quadrangle}

We introduce next a function on the manifold of quadrangles that measures what we call the aspect of a quadrangle and  is invariant under   similarity. This invariant is used  in the next section to describe  the moduli space $\Q/G$, when $G$ is a Nash subgroup of $\operatorname{Sim}(\R^2)$.

\begin{definition}
{Let $Q$ be a quadrangle with quadrangle form $[a:b:c]$, and let $\lambda_1$ and $\lambda_2$ be the eigenvalues of $M(a,b,c)$. The {\it (signed) aspect} of $Q$ is
\begin{eqnarray}\label{asp eq}
\alpha(Q)
\:\:=\:\:
\sgn(\lambda_1\lambda_2)\,\frac{\min\{|\lambda_1|,|\lambda_2|\}}
{\max\{|\lambda_1|,|\lambda_2|\}}
\:\:\in\:\:[-1,1].
\end{eqnarray}}
\end{definition}

{The definition is independent of the choice of representative of the projective class $[a:b:c]$. Since the eigenvalues of $M(a,b,c)$ are
$$\lambda_1=a+c-\sqrt{(a-c)^2+4b^2} \: {\text{ and }} \: \lambda_2=a+c+\sqrt{(a-c)^2+4b^2},$$
the aspect can also be expressed as
$$\alpha(Q) \:\:=\:\: \frac{|a+c|-\sqrt{(a-c)^2+4b^2}} {|a+c|+\sqrt{(a-c)^2+4b^2}}.$$
The denominator is nonzero since not all of $a,b,c$ are zero.}

{Thus the aspect is positive, negative, or zero according as the associated matrix $M(a,b,c)$ is definite, indefinite, or singular. Its absolute value is the ratio between the absolute values of the two eigenvalues of this matrix. When the center locus is an ellipse, its minor-to-major semiaxis ratio is $\sqrt{\alpha(Q)}$.}

\begin{lemma} \label{asp inv} The aspect of a quadrangle is invariant under the group of  similarities of the plane.
\end{lemma}

\begin{proof} By Remark~\ref{form remark}, the form of $Q$ is unaffected by  translations and uniform scaling, so we need only consider the effect of an orthogonal transformation $R$, which we view as a matrix. Let $[a:b:c]$ be the form of $Q$. Then by Remark~\ref{form remark} the form $[a':b':c']$ of the image of $Q$ under  $R$ is given by  $$M(a',b',c') = RM(a,b,c)R^T.$$  Since the eigenvalues of this matrix are the same as those of $M(a,b,c)$, the lemma follows from (\ref{asp eq}).\end{proof}

\begin{lemma}\label{mu theorem}
{After an orientation-preserving isometry of the plane, every quadrangle $Q$ has form $[\alpha(Q):0:1].$ Conversely, if $Q$ is similar to a quadrangle with form $[\alpha:0:1]$ for some $\alpha\in[-1,1]$, then $\alpha=\alpha(Q)$.}
\end{lemma}

\begin{proof}
Let $[a:b:c]$ be the form of $Q$, and let $\lambda_1,\lambda_2$ be the eigenvalues of $M(a,b,c)$. By the principal axis theorem, an orientation-preserving orthogonal change of coordinates diagonalizes this matrix. If necessary, a quarter turn interchanges the two eigendirections. After rescaling the resulting representative of the form, equation~(\ref{asp eq}) shows that the diagonal form is $[{\alpha}(Q):0:1]$. Conversely, if a quadrangle $Q'$ has form $[\alpha:0:1]$ with $-1\leq\alpha\leq1$, then
$${\alpha}(Q') =\frac{|1+\alpha|-|1-\alpha|}{|1+\alpha|+|1-\alpha|} =\alpha.$$
Aspect is invariant under similarities by Lemma~\ref{asp inv}, so the same holds for $Q$.
\end{proof}

\begin{theorem}\label{convex}\label{TFAE}
Let $Q$ be a quadrangle with aspect $\alpha={\alpha}(Q)$ and form $[a:b:c]$.
\begin{itemize}
\item[$(1)$] $\alpha>0 \Longleftrightarrow$ the {center }locus of $Q$ is an ellipse.
\item[$(2)$] $\alpha<0 \Longleftrightarrow$ the {center }locus of $Q$ is a hyperbola.
\item[$(3)$] $\alpha=0 \Longleftrightarrow$ the {center }locus of $Q$ is a pair of distinct parallel lines $\Longleftrightarrow$ $Q$ is a triple-side quadrangle.
\item[$(4)$] $\alpha=1 \Longleftrightarrow \operatorname{form}(Q)=[1:0:1]$.
\item[$(5)$] $\alpha=-1 \Longleftrightarrow \operatorname{form}(Q)=[a:b:-a]$ for some $a,b\in\R$, not both zero.
\end{itemize}
\end{theorem}

The possibilities in parts (1)--(3) are illustrated in Figure~\ref{loci figures}.
\begin{figure}[ht]
\begin{center}
\includegraphics[width=\textwidth,scale=.03]{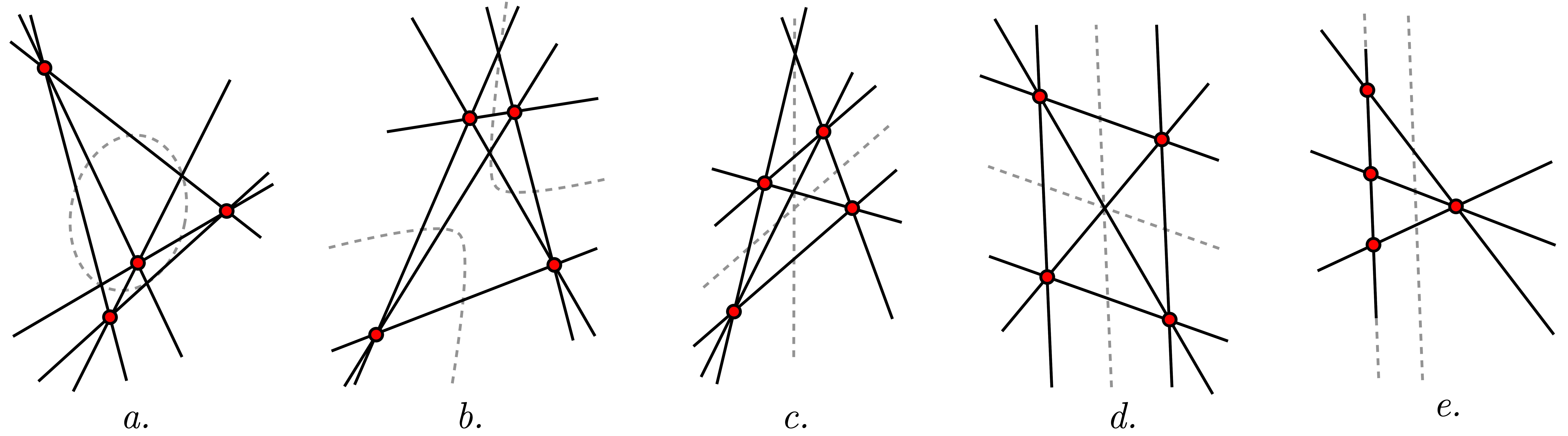}
\end{center}
\caption{The dashed curves are the {center }loci of the quadrangles shown. (a) If $\alpha(Q)>0$, the {center }locus is an ellipse. (b)--(d) If $\alpha(Q)<0$, the {center }locus is a hyperbola; in (c) and (d), it is a degenerate hyperbola. (e) If $\alpha(Q)=0$, the {center }locus is a pair of distinct parallel lines and $Q$ is a triple-side quadrangle.}
\label{loci figures}
\end{figure}

\begin{proof}
Choose two pairs of opposite sides $\{A,A'\}$ and $\{B,B'\}$ whose union contains a {nondegenerate} triangle, which the proof of Theorem~\ref{form submersion} shows is always possible. For each of these four sides $L$, choose an equation whose homogeneous part is $t_LX+u_LY$. For these choices, let $F$ and $G$ be the resulting line-pair polynomials, and write $J(F,G)_2=2aX^2-4bXY+2cY^2.$ Then $(a,b,c)$ represents $\operatorname{form}(Q)$. A direct calculation {shows that}
$$\begin{aligned} \det M(a,b,c)=-4& (t_Au_B-t_Bu_A)(t_Bu_{A'}-t_{A'}u_B)\\ &\mathrel{}\cdot(t_{A'}u_{B'}-t_{B'}u_{A'}) (t_{B'}u_A-t_Au_{B'}). \end{aligned}$$
Each of these factors involves the normals of two adjacent sides among $A,B,A',B'$. Since adjacent sides meet, a factor vanishes if and only if the corresponding sides coincide. Suppose, for example, that $A=L_{12}=L_{13}=B$. If $v_2=v_3$, then $v_4\ne v_2$, since otherwise all six sides would be concurrent. Hence $L_{34}=L_{24}$, so the two selected pairs of opposite sides are equal, contrary to their choice. Hence $v_2\ne v_3$, and the side $L_{23}$ is also the line $A=B$. Thus three adjacent sides coincide and $Q$ is a triple-side quadrangle. The other three factors are handled in the same way. Conversely, the two products associated with any two distinct pairs of opposite sides of a triple-side quadrangle share the triple side, so one of the four factors in the determinant vanishes. Thus $M(a,b,c)$ is singular if and only if $Q$ is a triple-side quadrangle.

Suppose that $Q$ is a triple-side quadrangle. After a   change of coordinates, its triple side is $Y=0$, its remaining vertex is $(0,1)$, and line-pair polynomials for two distinct pairs of opposite sides can be written, after rescaling, as
$$F=Y\bigl(X-m_1(Y-1)\bigr) \: {\mbox{ and }} \: G=Y\bigl(X-m_2(Y-1)\bigr),$$
where $m_1\ne m_2$. Direct calculation {shows that}
$$J(F,G)=(m_2-m_1)Y(1-2Y),$$
so the {center }locus of $Q$ is the pair of distinct parallel lines $Y=0$ and $Y=1/2$. Conversely, the homogeneous degree-two part of an equation for two parallel lines is the square of a linear form and hence has singular matrix. Since $\alpha=0$ exactly when $M(a,b,c)$ is singular, this proves (3).

It remains to consider the nonsingular case. The points at infinity of the projective realization of the {center-}locus conic {are the 
zeros} in $\P^1$ of $$(X,Y)M(a,b,c)(X,Y)^T.$$ This quadratic form has no real projective zeros when $M(a,b,c)$ is definite and two distinct real projective zeros when it is indefinite. By the standard affine classification of real conics, the {center }locus is, in the definition case, an ellipse a single point, or empty, while in the indefinite case it is a hyperbola.

The cases of a single-point or empty {center }locus do not occur because by Proposition~\ref{form coefficients}, every center of either selected line-pair conic belongs to the {center }locus. If either selected pair {consists of two distinct parallel lines or is a double line}, it has a whole line of centers, {namely the midline in the first case and the common line in the second}. Otherwise, 
each selected pair consists of two nonparallel lines and has their intersection as its center. These two intersection points are distinct, since otherwise all four selected sides would be concurrent, contrary to the choice that their union contains a {nondegenerate} triangle. Thus the {center }locus contains more than one point. Since the sign of $\alpha$ is the sign of the product of the two eigenvalues of $M(a,b,c)$, this proves (1) and~(2).

Finally, $\alpha=1$ if and only if $(a-c)^2+4b^2=0$; if and only if $a=c$ and $b=0$, which proves (4). Also, $\alpha=-1 $ if and only if  the eigenvalues of $M(a,b,c)$ sum to zero; if and only if $a+c=0$. This proves (5).
\end{proof}

{The next invariant is a   reparametrization of aspect that unlike aspect  remains Nash on all of $\Q$.}

\begin{definition}
{Let $Q$ be a quadrangle with form $[a:b:c]$, and set $M=M(a,b,c)$. The {\it matrix sphericity} of $Q$ is
$$\sph(Q) \:\:=\:\: \frac{(\tr(M))^2}{2\tr(M^2)}.$$}
\end{definition}

{If $\lambda_1$ and $\lambda_2$ are the eigenvalues of $M$, {then, since $M$ is symmetric, we have} the equivalent expressions
\begin{equation}\label{sphericity formulas}
\sph(Q)
=\frac{(a+c)^2}{2(a^2+2b^2+c^2)}
=\frac{(\lambda_1+\lambda_2)^2}{2(\lambda_1^2+\lambda_2^2)}
=\frac{(1+\alpha(Q))^2}{2(1+\alpha(Q)^2)}.
\end{equation}
This value equals $1$ if and only if $M$ is scalar, i.e.,   $a=c$ and $b=0$.}

Note that $\sph(Q)$ is independent of the choice of  representative of the form of~$Q$. Also, sphericity takes values in $[0,1]$ since $$2(1+\alpha(Q)^2)-(1+\alpha(Q))^2=(1-\alpha(Q))^2.$$

\begin{remark} \label{sph calc} By equation~(\ref{sphericity formulas}), aspect is expressible in terms of  sphericity:  $${\alpha}(Q) \:\: =\:\:  \frac{\sqrt{\sph(Q)}- \sqrt{1-\sph(Q)}}{\sqrt{\sph(Q)}+ \sqrt{1-\sph(Q)} }.$$
\end{remark}

\begin{proposition} \label{sphericity}  Sphericity is a   Nash function on $\Quad$ that is invariant under   similarity. For any Nash subgroup $G$ of $\operatorname{Sim}(\R^2)$, sphericity descends to a   Nash function on $\Quad/G$.
\end{proposition}

\begin{proof}
The homogeneous expression for $\sph$ in equation~(\ref{sphericity formulas}) is well-defined on $\P^2$ and is Nash in each affine chart.  The function $\sph:\Quad \rightarrow \R$ is a Nash function on  $\Quad$ since the denominator does not vanish and the form map on $\Q$ is a Nash map by  Theorem~\ref{form submersion}.   Remark~\ref{form remark} implies that sphericity  is invariant under similarities since the eigenvalues of similar matrices are the same and eigenvalues scale uniformly under uniform scaling of the plane. Let $\pi:\Quad \rightarrow \Quad/G$ denote the quotient map, which  by Theorem~\ref{main Nash} is a Nash submersion whose fibers are the $G$-orbits. Since $\sph$ is $G$-invariant, there is a unique map $\sigma:\Quad/G\rightarrow\R$ such that $\sph=\sigma\circ\pi$. The quotient map $\pi$ is a surjective Nash submersion, and $\sph$ is Nash, so Lemma~\ref{Nash descent} shows that $\sigma$ is Nash. Thus sphericity descends to a Nash function on $\Quad/G$.
\end{proof}

We show next that sphericity is a submersion as long as the endpoints of its range $[0,1]$ are avoided. By Remark~\ref{sph calc}, these endpoints correspond to aspect values  $\pm 1$.

\begin{lemma} \label{sph submersion} The sphericity function   restricts to a Nash submersion on  the open Nash submanifold of $\Q$ consisting of the quadrangles  with aspect $\ne \pm 1$.
\end{lemma}

\begin{proof}  Let $\Q^*$ denote the set of quadrangles with aspect $\ne \pm 1$. By Theorem~\ref{TFAE}, a quadrangle lies in $\Q^*$ if and only if its form $[a:b:c]$ satisfies $a\ne -c$ as well as $b\ne 0$ or $a \ne c$. Thus $\Q^*={\operatorname{form}}^{-1}(U)$, where $$U = \{[a:b:c] \in \P^2\mid a\ne -c\} \cap \{[a:b:c] \in \P^2\mid b \ne 0 {\mbox{ or }} a \ne c\}$$ is an open subset of $\P^2$. Define a function $f:\R^2 \setminus \{(1/2,0)\} \rightarrow  \R$ by $$f(x,y) = \frac{1}{4x^2+4y^2-4x+2}.$$ The denominator in this expression is never $0$ since $$4x^{2}  + 4y^{2} - 4x + 2 = 4(x-\tfrac12)^{2} + 4y^{2} + 1 \ge 1,$$ so $f$ is a Nash function. Also, $$\nabla f(x,y) = \left( \frac{4 - 8x}{\left(4x^{2}+4y^{2}-4x+2\right)^{2}}, \frac{-8y}{\left(4x^{2}+4y^{2}-4x+2\right)^{2}} \right),$$ so $\nabla f(x,y) \ne  0$ since by assumption $(x,y) \ne (1/2,0)$.       Thus $f$ is a Nash submersion. Consider the Nash diffeomorphism $$g:U \longrightarrow \R^2\setminus \{(1/2,0)\}:[a:b:c] \longmapsto \left(\frac{a}{a+c},\frac{b}{a+c}\right)$$ whose inverse is $$g^{-1}:\R^2\setminus \{(1/2,0)\} \longrightarrow U:(x,y)\longmapsto [x:y:1-x].$$ Then for each $Q \in \Quad^*$, a calculation using equation~(\ref{sphericity formulas}) shows that $$f(g({\operatorname{form}}(Q))) = f\left(\frac{a}{a+c},\frac{b}{a+c}\right) = \frac{(a+c)^2}{2(a^2+2b^2+c^2)} = \sph(Q).$$ Since the maps $f$ and $g$ are Nash submersions and  the form map is a Nash submersion by Theorem~\ref{form submersion}, the restriction of the sphericity map   to $\Q^*$ is a Nash submersion.
\end{proof}

 For lack of an explicit reference for  the following lemma, we include a proof.

\begin{lemma} \label{Nash quotient}
Let $X$ be an embedded Nash submanifold of a Nash manifold $Y$, and let $G$ be a Nash group acting Nashly, properly and freely on $Y$ such that $X$ is $G$-invariant. Then $X/G$ is an embedded Nash submanifold of the Nash manifold $Y/G$.

\end{lemma}

\begin{proof}
The restricted action of $G$ on $X$ is Nash and free. It is proper by the sequential criterion \cite[Proposition~21.5, p.~543]{Lee}: if $x_j$ and $g_j\cdot x_j$ converge in $X$, they also converge in $Y$, so properness of the action on $Y$ implies that $(g_j)$ has a convergent subsequence. Lemma~\ref{Nash lemma} applies to both $X$ and $Y$. Thus $X/G$ and $Y/G$ are Nash manifolds, and the quotient maps $\pi_X:X\to X/G$ and $ \pi_Y:Y\to Y/G $ are Nash submersions.  Since $X$ is $G$-invariant, $\pi_Y^{-1}(\pi_Y(X))=X$, so the map $i:X/G\to Y/G$ induced by the inclusion $X\subseteq Y$ is injective, and its image is $\pi_Y(X)$. Since $i\circ\pi_X=\pi_Y|_X$ and $\pi_X$ is a surjective Nash submersion, Lemma~\ref{Nash descent} shows that $i$ is Nash, and hence continuous. We verify that $i$ is open onto $\pi_Y(X)$. Let $O\subseteq X/G$ be open. Then $\pi_X^{-1}(O)$ is open in $X$, so there is an open set $V\subseteq Y$ such that $V\cap X=\pi_X^{-1}(O).$ Since $X$ is $G$-invariant, $\pi_Y(V)\cap\pi_Y(X)=\pi_Y(V\cap X)=i(O)$. The quotient map $\pi_Y:Y\to Y/G$ is open, so $\pi_Y(V)\cap\pi_Y(X)$ is open in the subspace $\pi_Y(X)$. Thus $i:X/G\to\pi_Y(X)$ is an open continuous bijection and hence a homeomorphism. It remains to show that $i$ is a Nash embedding.

We next show that $i$ is an immersion. Fix $p\in X/G$ and choose $x\in X$ such that $\pi_X(x)=p$. Since $X$ is $G$-invariant, $T_x(G\cdot x)\subseteq T_xX$.  Since $\pi_X$ and $\pi_Y$ are submersions, the submersion level-set theorem \cite[Corollary~5.13, p.~105, and Proposition~5.38, p.~117]{Lee} identifies the kernel of each differential at $x$ with the tangent space to the corresponding fiber. The fibers through $x$ are both the orbit $G\cdot x$, so
\begin{equation}\label{quotient kernels}
\ker(d\pi_X)_x=T_x(G\cdot x)=\ker(d\pi_Y)_x.
\end{equation}
{Differentiating} $i\circ\pi_X=\pi_Y|_X$ at $x${, we have}
\begin{equation}\label{quotient chain rule}
(di)_p\circ(d\pi_X)_x=(d\pi_Y)_x|_{T_xX}.
\end{equation}
Suppose $v\in T_p(X/G)$ and $(di)_p(v)=0$. Because $(d\pi_X)_x$ is surjective, we can choose $w\in T_xX$ such that $(d\pi_X)_x(w)=v$. Equations~\eqref{quotient kernels} and~\eqref{quotient chain rule} imply $$w\in\ker(d\pi_Y)_x=\ker(d\pi_X)_x,$$ so $v=0$. Thus $i$ is an immersion and, being a homeomorphism onto its image, a smooth embedding.

The image $S=\pi_Y(X)$ is semialgebraic by the Tarski-Seidenberg theorem \cite[Proposition~2.2.7, p.~29]{BCR} and is a smooth embedded submanifold by the preceding argument. Hence it has the induced Nash structure by \cite[Proposition~2.9]{Sun}. With its codomain restricted to $S$, the map $i$ is Nash and a smooth diffeomorphism, hence a surjective Nash submersion. Applying Lemma~\ref{Nash descent} to $i^{-1}\circ i=\operatorname{id}_{X/G}$ shows that $i^{-1}$ is Nash. Therefore, $i$ is a Nash embedding.
\end{proof}

\begin{theorem} Let $G$ be a Nash subgroup of $\operatorname{Sim}(\R^2)$.  For each $\alpha \in [-1,1]$, the orbit space $$\{\: {\mbox{\rm aspect } } \alpha {\mbox{\rm{ quadrangles}}}\:\}/G$$ is a closed embedded  Nash submanifold of $\Quad/G$ whose codimension is $1$ if $\alpha \ne 1$ and~$2$ if $\alpha = 1$.

\end{theorem}

\begin{proof}   By equation~(\ref{sphericity formulas}), if a quadrangle $Q$ has aspect $\alpha$, then $$\sph(Q)=\frac{(1+\alpha)^2}{2(1+\alpha^2)}.$$   Conversely, set $r_\alpha=\sph(Q)$. Remark~\ref{sph calc} recovers the aspect from the sphericity, so the function $\alpha\mapsto r_\alpha$ is one-to-one on $[-1,1]$. Thus the aspect $\alpha$ quadrangles are exactly the quadrangles with sphericity $r_\alpha$. To see that every $\alpha\in[-1,1]$ occurs, consider the four vertices $v_1=(0,0),$ $v_2=(1,\alpha),$  $v_3=(1,-2),$ $v_4=(-1,0).$ For $\alpha\in[-1,1]$, these points are distinct and not all collinear. Choose three noncollinear vertices among them. Their three joining lines form a {nondegenerate} triangle, so the six joining lines define a quadrangle by Definition~\ref{quadrangle def}. The two pairs of opposite sides $A=L_{12},A'=L_{34}$ and $B=L_{13},B'=L_{24}$ have respective normals $[-\alpha:1], [1:1], [2:1], [-\alpha/2:1].$ {By Equation~(\ref{affine form coefficients}),}
$$(a,b,c)=\left(-\frac{\alpha(2+\alpha)}2,0,-\frac{2+\alpha}2\right).$$
Since $2+\alpha\ne0$, this quadrangle has form $[\alpha:0:1]$ and hence aspect $\alpha$.

We now prove the submanifold assertion, beginning with the sphericity values in $(0,1)$. By Remark~\ref{sph calc}, the quadrangles with these sphericity values {form the set} $\Q^*$ of quadrangles with aspect $\ne \pm 1$ from Lemma~\ref{sph submersion}. Let $r \in (0,1)$, and let~$\Q_r$ be the set of quadrangles with sphericity $r$. The sphericity map restricted to $\Q^*$  is a submersion by Lemma~\ref{sph submersion}, so the preimage of $\{r\}$ in~$\Q^*$ is an embedded submanifold of codimension~$1$ \cite[Corollary~5.13, p.~105]{Lee}. Since $r\in(0,1)$, this preimage is exactly $\Q_r$, the preimage of $\{r\}$ in~$\Q$ under the sphericity function. The latter description also shows that $\Q_r$ is closed in~$\Q$. Also, since $\{r\}\subseteq \R$ is a semialgebraic set, and preimages of semialgebraic sets under semialgebraic maps are semialgebraic \cite[Proposition~2.2.7, p.~29]{BCR}, $\Q_r$ is a semialgebraic subset of $\Q$, and so $\Q_r$ is a closed embedded Nash submanifold of $\Q$ of codimension~$1$ for all $r \in (0,1)$.

Consider the case $\alpha =-1$, which is sphericity $0$. Let $L$ be the line  in $\P^2$ given by $L = \{[a:b:-a] \mid a,b \in \R {\mbox{ not both }} 0\}.$ By Theorem~\ref{TFAE}, the set $\Q_0$ of quadrangles with sphericity $0$ (aspect $-1$) is the preimage of $L$ under the form map. {By \cite[Proposition~5.16, p.~106]{Lee}, each point of $L$ has a neighborhood} $V$ in $\P^2$ and a submersion $h:V\to\R$ such that $L\cap V=h^{-1}(0)$. Since the form map is a submersion by Theorem~\ref{form submersion}, the composition of $h$ with the form map is a submersion on ${\operatorname{form}}^{-1}(V)$. Thus \cite[Corollary~5.13, p.~105]{Lee} shows locally that $\Q_0$ is an embedded submanifold of $\Q$ of codimension~$1$. Since $L$ is closed and semialgebraic and the form map is Nash, $\Q_0$ is a closed semialgebraic subset of $\Q$, and hence is a closed embedded Nash submanifold.

Finally, consider the case $\alpha =1$, which is sphericity $1$. By Theorem~\ref{TFAE}, the set~$\Q_1$ of quadrangles with sphericity $1$ is the preimage of $[1:0:1]$ under the form map. Since the form map is a submersion, \cite[Corollary~5.13, p.~105]{Lee} shows that $\Q_1$ is an embedded submanifold of $\Q$ of codimension $2$. It is closed and semialgebraic because it is the preimage of a point under a Nash map, and hence it is a closed embedded Nash submanifold.

We have shown that for each possible sphericity $r \in [0,1]$, the set $\Q_r$ of quadrangles with sphericity $r$ is a closed embedded Nash submanifold that has codimension~$1$ if $r \ne 1$ and codimension~$2$ if $r = 1$. If $G$ is a  Nash subgroup of $\operatorname{Sim}(\R^2)$, then $\Q_r$ is $G$-invariant by Proposition~\ref{sphericity}, and so $\Q_r/G$ is an embedded Nash submanifold of $\Q/G$ by Lemma~\ref{Nash quotient}. Since $\Q_r$ is $G$-invariant, it is saturated for the quotient map $\pi:\Q\to\Q/G$, and hence $\pi^{-1}(\pi(\Q_r))=\Q_r$. Since $\Q_r$ is closed in $\Q$, the definition of the quotient topology implies that $\pi(\Q_r)=\Q_r/G$ is closed in $\Q/G$.  The codimension is unchanged after quotienting because the action on $\Q_r$ is still free and proper, so $\dim(\Q_r/G)=\dim \Q_r-\dim G $ and $ \dim(\Q/G)=\dim \Q-\dim G.$ Thus $\Q_r/G$ and $\Q/G$ differ by the same codimension as $\Q_r$ and $\Q$.  This proves the theorem.
 \end{proof}

\section{The moduli space of \texorpdfstring{{nonexceptional}}{nonexceptional} quadrangles}

We   use the aspect function to study the moduli space $\Q/G$ of quadrangles modulo  Nash subgroups $G$ of $\operatorname{Sim}(\R^2)$. To do so, we need to restrict to what we will call {nonexceptional} quadrangles.

\begin{definition} A quadrangle $Q$ is {\it {nonexceptional}} if its aspect is not $-1$, $0$ or $1$.
\end{definition}

  As we show in Lemma~\ref{asp is nash}, one reason for restricting to nonexceptional quadrangles is that the aspect function is Nash on these quadrangles even though it is not Nash on all of $\Q$. There are other issues too. At aspect $1$, a matrix representing the form is scalar, so it has no distinguished pair of eigendirections. At aspect $-1$, the two eigenvalues have equal absolute value and opposite signs, producing a quarter-turn ambiguity in choosing an eigenframe. Aspect $0$ separates definite from indefinite forms, and the boost construction in Theorem~\ref{hard} applies only to two nonzero aspect values of the same sign. So for this section we focus on the quadrangles whose aspect is not one of these three problematic values.

Taking the set of {nonexceptional} quadrangles amounts to removing three fibers from the aspect map from $\Q$ to $[-1,1]$. Off the fibers we have   a Nash manifold:

 \begin{lemma}
Let $G$ be a  Nash subgroup of $\operatorname{Sim}(\R^2)$.  The set of {nonexceptional} quadrangles is an open Nash submanifold of $\Quad$ whose $G$-orbit space is an open Nash submanifold of $\Quad/G$.

 \end{lemma}

  \begin{proof}
By Proposition~\ref{sphericity}, sphericity is a Nash function on $\Quad$, so it is semialgebraic. Thus the open subset $$U=\sph^{-1}((0,1/2) \cup (1/2,1)) \subseteq \Quad$$ is  semialgebraic and  hence a Nash submanifold of $\Quad$. By equation~(\ref{sphericity formulas}), sphericity is $0,$ $1/2$ or $1$ if and only if aspect is $-1,$ $0$ or $1$. Hence this submanifold is the set of {nonexceptional} quadrangles, so the {nonexceptional} quadrangles form a Nash submanifold of $\Quad$.

Since sphericity is invariant under similarities by Proposition~\ref{sphericity}, the set $U$ of {nonexceptional} quadrangles is $G$-invariant, and so $U$ is saturated with respect to the quotient map $\pi:\Quad\to\Quad/G$. Because $U$ is open in $\Quad$ and $\pi$ is a quotient map, the image $\pi(U)$ is open in $\Quad/G$. Moreover, $\pi$ is a Nash map (Theorem~\ref{main Nash}), hence semialgebraic. Since $U$ is semialgebraic, it follows from the Tarski-Seidenberg theorem \cite[Proposition~2.2.7, p.~29]{BCR} that $\pi(U)$ is also semialgebraic. Therefore, $\pi(U)$ is an open semialgebraic subset of the Nash manifold $\Quad/G$, and hence a Nash submanifold. This submanifold {is the} $G$-orbit space of the {nonexceptional} quadrangles.
\end{proof}

In Theorem~\ref{hard}, we work with the fibers of the aspect map, and to ensure that these fibers vary Nashly across the non-excluded aspect ratios, we need that the aspect map is a Nash function on the {nonexceptional} quadrangles, which the next lemma shows is the case.

\begin{lemma} \label{asp is nash}
Let $\Q^*$ be the subset of $\Q$ consisting of the quadrangles with aspect in $(-1,1)$, and let $G$ be a Nash subgroup of $\operatorname{Sim}(\R^2)$. Then $\Q^*$ is an open Nash submanifold of $\Q$, the restriction of the aspect map to $\Q^*$ is a Nash function that descends to a Nash function on $\Q^*/G$, and this restriction has no Nash extension to all of $\Q$.
\end{lemma}

\begin{proof} By Remark~\ref{sph calc}, $\Q^*=\sph^{-1}((0,1))$, so Proposition~\ref{sphericity} shows that $\Q^*$ is an open Nash submanifold of $\Q$. On $(0,1)$, the expression in Remark~\ref{sph calc} is a Nash function of sphericity, and hence the aspect map is Nash on $\Q^*$. Since sphericity descends to a Nash function on $\Q/G$ by Proposition~\ref{sphericity}, the same expression shows that aspect descends to a Nash function on the open Nash submanifold $\Q^*/G$ of $\Q/G$. To prove that this restriction has no Nash extension (or even an analytic extension) to $\Quad$, for each $t\in (-1,1)$ let $s(t)=\sqrt{1-t}$ and let $Q_t$ be the parallelogram for which two of its pairs of  opposite sides are $$\{X+s(t)Y=\pm 1\} \: {\text{ and }} \: \{X-s(t)Y=\pm 1\}.$$ The four vertices of $Q_t$ are
$$v_1=(1,0),\:\: v_2=(0,1/s(t)),\:\: v_3=(0,-1/s(t)),\:\: v_4=(-1,0),$$
and the remaining two sides are the coordinate axes. Since $s(t)>0$ and $s(t)$ is analytic on $(-1,1)$, all the vertices and normals of $Q_t$ vary analytically with $t$. Hence $t\mapsto Q_t$ defines an analytic map $\gamma:(-1,1)\to \Q.$ By equation~(\ref{affine form coefficients}), applied to the two displayed pairs of opposite sides, the form of $Q_t$ is $[1:0:t-1]$. Hence $$(\alpha\circ\gamma )(t) = \frac{|t|-(2-t)}{|t|+(2-t)} = \begin{cases} t-1,& t>0,\\ \dfrac{1}{t-1},& t<0. \end{cases}$$ Thus $$\lim_{t\to 0^+}(\alpha\circ\gamma )'(t)=1 \:\mbox{ and }\: \lim_{t\to 0^-}(\alpha\circ\gamma )'(t)=-1.$$ Therefore, $\alpha\circ\gamma $ is not analytic at $0$. For $t\ne0$, the displayed formula shows that $Q_t\in\Q^*$. If the restriction of aspect to $\Q^*$ had an analytic extension to all of $\Q$, its composition with $\gamma$ would be analytic and would agree with $\alpha\circ\gamma$ for $t\ne0$. By continuity it would also agree at $t=0$, a contradiction.
\end{proof}

The next theorem shows that for any Nash subgroup $G$ of $\operatorname{Sim}(\R^2)$, the fibers of the descended aspect map on $\Q/G$ are Nash diffeomorphic over $(-1,0)$, and similarly the fibers are Nash diffeomorphic over $(0,1)$. To prove this theorem, it will be convenient to trace-normalize the matrix $M(a,b,c)$ of the form of a {nonexceptional} quadrangle, which we do along the following lines.

Theorem~\ref{TFAE} implies that the form $[a:b:c]$ of a {nonexceptional} quadrangle occurs in the affine chart $a+c \ne 0$ of $\P^2$, and so we can represent the form and its eigenvalues in a canonical way.

\begin{notation}
Let $Q$ be a {nonexceptional} quadrangle with form $[a:b:c]$. Since $Q$ is {nonexceptional}, $a+c \ne 0$,  and we let $N(Q)$ be the normalized form $$N(Q)  \:\:= \:\: \frac{1}{2(a+c)}M(a,b,c) \:\: = \:\:  \frac{1}{a+c}\begin{pmatrix} a & -b \\ -b & c\end{pmatrix}.$$ 
\end{notation}

The matrix $N(Q)$ is independent of the choice of representative of the form $[a:b:c]$ of $Q$ and has trace $1$.

\begin{lemma}\label{normalized eigenvalues}
Let $Q$ be a nonexceptional quadrangle, and let $\lambda_1(Q)$ and $\lambda_2(Q)$ be the eigenvalues of $N(Q)$, labeled so that $|\lambda_1(Q)|<|\lambda_2(Q)|$. Then
\begin{equation}\label{alpha}
\alpha(Q)=\frac{\lambda_1(Q)}{\lambda_2(Q)}.
\end{equation}
Consequently, the eigenvalues depend only on the aspect of $Q$ and satisfy
$$\lambda_1(Q)=\frac{\alpha(Q)}{1+\alpha(Q)} \quad\mbox{and}\quad \lambda_2(Q)=\frac{1}{1+\alpha(Q)}.$$
\end{lemma}

\begin{proof}
Let $[a:b:c]$ be the form of $Q$. Since $N(Q)$ is a nonzero scalar multiple of $M(a,b,c)$, passing from $M(a,b,c)$ to $N(Q)$ does not change the signed ratio of the absolute values of the eigenvalues. Since $\alpha(Q)\ne\pm1$, these absolute values are distinct, so the eigenvalues can be labeled as stated. Thus, with $\lambda_i=\lambda_i(Q)$, the definition of aspect in equation~(\ref{asp eq})   implies
$$\alpha(Q)=\sgn(\lambda_1\lambda_2)\frac{|\lambda_1|}{|\lambda_2|} =\frac{\lambda_1}{\lambda_2}.$$
Since $\tr N(Q)=1$, we have $1=\lambda_1+\lambda_2=(1+\alpha(Q))\lambda_2.$ Since $\alpha(Q)\ne-1$, solving for $\lambda_2$ and using $\lambda_1=\alpha(Q)\lambda_2$ proves the stated formulas.
\end{proof}

We show next that if \(\alpha\) and \(\beta\) both lie in \((-1,0)\) or both lie in \((0,1)\), then the corresponding orbit spaces are Nash diffeomorphic. The idea behind the proof is to translate the centroid to the origin, align the principal axes with the coordinate axes,  boost along the axes that change the aspect from $\alpha$ to $\beta$ (as in Figure~\ref{double squeeze quadrangles}), and then restore the original coordinates.   The challenge here is that for $\alpha \ne \beta$, this process cannot be accomplished by a single affine transformation  for all quadrangles of aspect $\alpha$,  and so must be done for each quadrangle individually in a way that varies Nashly across the across the space of aspect $\alpha$ quadrangles and descends also to the orbit space.

 \begin{theorem} \label{hard} Let $\alpha,\beta \in (-1,0) \cup (0,1)$ with the same sign, and let $G$ be a  Nash subgroup of $\operatorname{Sim}(\R^2)$. There  is a Nash diffeomorphism $$\{\: {\mbox{\rm aspect } } \alpha {\mbox{\rm{ quadrangles}}}\:\}/G \:\: \simeq \:\: \{\: {\mbox{\rm aspect } } \beta {\mbox{\rm{ quadrangles}}}\:\}/G.$$

\end{theorem}

\begin{proof}
For $\delta\in(-1,0)\cup(0,1)$, let $\Quad(\delta)$ denote the set of quadrangles of aspect~$\delta$. By Lemma~\ref{normalized eigenvalues}, the eigenvalues of $N(Q)$ for $Q\in\Quad(\delta)$ are
$$\lambda_1(\delta)=\frac{\delta}{1+\delta} \quad\mbox{and}\quad \lambda_2(\delta)=\frac{1}{1+\delta}.$$
For $i=1,2$, denote the $\lambda_i(\delta)$-eigenspace of $N(Q)$ by
$$\mathcal E_i(Q)=\ker\big(N(Q)-\lambda_i(\delta)I\big).$$
  Since $\delta=\alpha(Q)\ne1$, we can define a $2 \times 2$ matrix
\begin{equation}\label{projection formula}
P(Q)
=\frac{N(Q)-\lambda_2(\delta)I}
{\lambda_1(\delta)-\lambda_2(\delta)}
=\frac{(1+\alpha(Q))N(Q)-I}{\alpha(Q)-1}.
\end{equation}
If $v_i\in\mathcal E_i(Q)$, then
$$P(Q)v_1 =\frac{\lambda_1(\delta)-\lambda_2(\delta)} {\lambda_1(\delta)-\lambda_2(\delta)}v_1=v_1 \:\: \mbox{and}\:\: P(Q)v_2 =\frac{\lambda_2(\delta)-\lambda_2(\delta)} {\lambda_1(\delta)-\lambda_2(\delta)}v_2=0.$$
Since $N(Q)$ is symmetric, its   eigenspaces are orthogonal, and
$$\R^2=\mathcal E_1(Q)\mathbin{\oplus}\mathcal E_2(Q).$$

\begin{figure}[tbp]
\centering
\includegraphics[width=.8\textwidth]{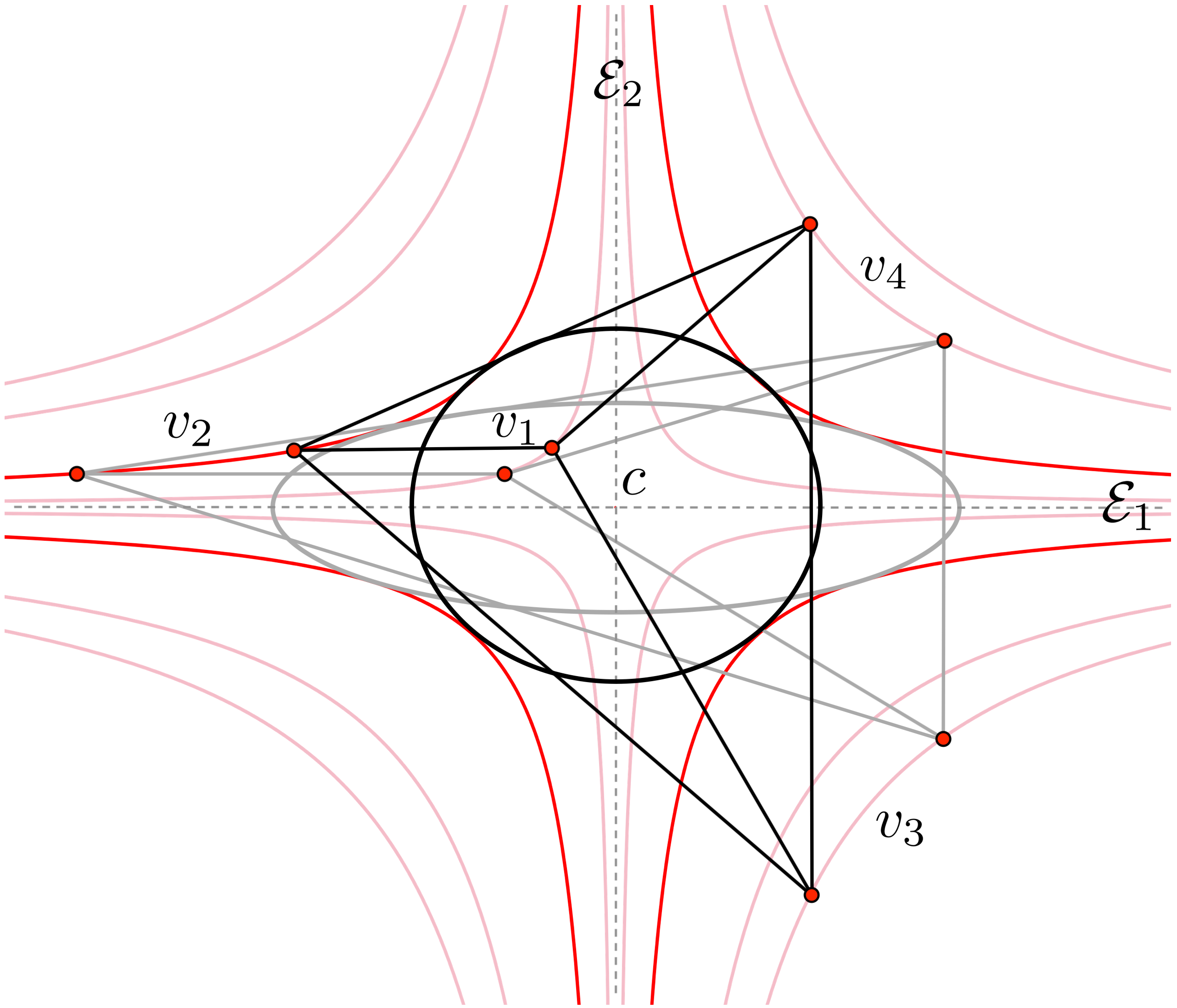}
\caption{A quadrangle on its principal axes before and after a boost along its principal axes. The vertices follow the paths of the hyperbolas. }
\label{double squeeze quadrangles}
\end{figure}

Thus, if $v=v_1+v_2$ with $v_i\in\mathcal E_i(Q)$, then $P(Q)v=v_1$.  Therefore, $P(Q)$ is the orthogonal projection onto $\mathcal E_1(Q)$, while $I-P(Q)$ is the orthogonal projection onto $\mathcal E_2(Q)$.  In particular, $P(Q)^T=P(Q)$ and $P(Q)^2=P(Q).$ Since $N$ and the aspect map are Nash on the set of {nonexceptional} quadrangles, equation~\eqref{projection formula} shows directly that $P(Q)$ depends Nashly on $Q$.  Since $P(Q)$ and $I-P(Q)$ are the orthogonal projections onto $\mathcal E_1(Q)$ and $\mathcal E_2(Q)$, respectively, the spectral decomposition of $N(Q)$ is
\begin{equation}\label{spectral decomposition}
N(Q)=\lambda_1(\delta)P(Q)+\lambda_2(\delta)(I-P(Q))
=\frac{1}{1+\delta}
\big(\delta P(Q)+I-P(Q)\big).
\end{equation}

Now fix $\alpha$ and $\beta$ as in the theorem.  Since they have the same sign, $\alpha/\beta>0$.  Let $s=({\alpha}/{\beta})^{1/4}.$ For $Q\in\Quad(\alpha)$, define a boost matrix by
\begin{equation}\label{boost matrix}
A_{\alpha,\beta}(Q)
=sP(Q)+s^{-1}\big(I-P(Q)\big).
\end{equation}
Since $P(Q)$ and $I-P(Q)$ are the projections onto $\mathcal E_1(Q)$ and $\mathcal E_2(Q)$, respectively, $A_{\alpha,\beta}(Q)$ acts by the factor $s$ on $\mathcal E_1(Q)$ and by the reciprocal factor $s^{-1}$ on $\mathcal E_2(Q)$.  Thus its principal directions agree with those of the form of $Q$. This matrix is symmetric and positive definite, with inverse
\begin{equation}\label{boost inverse}
A_{\alpha,\beta}(Q)^{-1}
=s^{-1}P(Q)+s\big(I-P(Q)\big).
\end{equation}
The positive fourth-root function is Nash on $(0,\infty)$. Since $s=(\alpha/\beta)^{1/4}$ and $P(Q)$ depends Nashly on $Q$, equations~(\ref{boost matrix}) and~(\ref{boost inverse}) show that $A_{\alpha,\beta}(Q)$ and $A_{\alpha,\beta}(Q)^{-1}$ depend Nashly on $Q$, $\alpha$, and $\beta$ on each of the two same-sign parameter regions in the statement of the theorem.

{Let $c(Q)$ be the centroid of $Q$, i.e., $c(Q)=(v_1+v_2+v_3+v_4)/4.$ } Define the affine transformation
$$f_Q(x)=c(Q)+A_{\alpha,\beta}(Q)(x-c(Q)),$$
so that  $f_Q$ is the boost about $c(Q)$ along the principal axes $c(Q)+\mathcal E_1(Q)$ and $c(Q)+\mathcal E_2(Q)$. Define $\Theta_{\alpha,\beta}(Q)=f_Q\cdot Q.$ The map $Q\mapsto f_Q$ is Nash because $Q\mapsto c(Q)$ and $Q\mapsto A_{\alpha,\beta}(Q)$ are Nash, so since the affine action on $\Quad$ is Nash, $\Theta_{\alpha,\beta}(Q)=f_Q\cdot Q$ is Nash. Also, $f_Q$ fixes $c(Q)$, so $\Theta_{\alpha,\beta}(Q)$ has the same centroid as $Q$.

Put $Q'=\Theta_{\alpha,\beta}(Q)$.  By Remark~\ref{form remark}, a matrix representing the form of $Q'$ is
$$L=A_{\alpha,\beta}(Q)^{-T}N(Q)A_{\alpha,\beta}(Q)^{-1}.$$
Since $A_{\alpha,\beta}(Q)$ is symmetric, $A_{\alpha,\beta}(Q)^{-T}=A_{\alpha,\beta}(Q)^{-1}$.  Using the spectral decomposition in equations~(\ref{spectral decomposition}) and~(\ref{boost inverse}), we obtain
$$\begin{aligned} L &=\frac{1}{1+\alpha} \big(s^{-1}P(Q)+s(I-P(Q))\big)\\ &\quad\cdot\big(\alpha P(Q)+I-P(Q)\big)\\ &\quad\cdot\big(s^{-1}P(Q)+s(I-P(Q))\big)\\ &=\frac{1}{1+\alpha} \big(\alpha s^{-2}P(Q)+s^2(I-P(Q))\big)\\ &=\frac{s^2}{1+\alpha} \big(\beta P(Q)+I-P(Q)\big), \end{aligned}$$
where the last equality uses $s^4=\alpha/\beta$.  Thus $L$ has eigenspaces $\mathcal E_1(Q)$ and $\mathcal E_2(Q)$, with corresponding eigenvalues
$$\frac{s^2\beta}{1+\alpha} \:\: \mbox{ and } \:\: \frac{s^2}{1+\alpha},$$
respectively.  Their ratio is $\beta$, and the first has smaller absolute value because $0<|\beta|<1$.  Their sum is
$$\tr(L)=\frac{s^2(1+\beta)}{1+\alpha}\ne0.$$
By equation~(\ref{asp eq}), $Q'$ has aspect $\beta$. In particular, $Q'$ is nonexceptional, and after trace normalization we have $$N(Q')=\frac{1}{1+\beta} \big(\beta P(Q)+I-P(Q)\big).$$ It follows at once from the definition of the projection that
$$P(Q') =\frac{(1+\beta)N(Q')-I}{\beta-1} =P(Q).$$
Thus the boost map $\Theta_{\alpha,\beta}$ sends $Q$ to a quadrangle of aspect $\beta$ while preserving its centroid and ordered principal axes.

We claim next that $\Theta_{\beta,\alpha}$ is the inverse of $\Theta_{\alpha,\beta}$. Since $(\beta/\alpha)^{1/4}=s^{-1}$ and $P(Q')=P(Q)$,
$$\begin{aligned} A_{\beta,\alpha}(Q') =s^{-1}P(Q)+s\big(I-P(Q)\big) =A_{\alpha,\beta}(Q)^{-1}. \end{aligned}$$
The centered affine transformations used in the two constructions have the same fixed centroid and inverse linear parts.  Hence $\Theta_{\beta,\alpha}(Q')=Q$.  {The reverse composition follows by interchanging $\alpha$ and $\beta$}, so $\Theta_{\alpha,\beta}:\Quad(\alpha)\rightarrow\Quad(\beta)$ is a Nash diffeomorphism with inverse $\Theta_{\beta,\alpha}$.

It remains to prove similarity equivariance.  Write an arbitrary similarity as ${g}(x)=rRx+p,$ where $r>0$, $R$ is orthogonal, and $p\in\R^2$.  The translation $p$ changes only lower-degree terms and hence does not affect the form.  Writing brackets for projective classes and setting $S=rR$, {we have, by Remark~\ref{form remark},}
$$[M({g}\cdot Q)] =[(rR)^{-T}M(Q)(rR)^{-1}] =[r^{-2}RM(Q)R^T].$$
Thus, choosing compatible matrix representatives, we may take
$$M({g}\cdot Q)=r^{-2}RM(Q)R^T.$$
Since $\tr(RM(Q)R^T)=\tr(M(Q))$, {after trace normalization we have}
$$N({g}\cdot Q) =\frac{r^{-2}RM(Q)R^T}{r^{-2}\tr(M(Q))}=RN(Q)R^T.$$
By Lemma~\ref{asp inv}, $\alpha({g}\cdot Q)=\alpha(Q)=\alpha$, so the same $s=(\alpha/\beta)^{1/4}$ is used for $Q$ and ${g}\cdot Q$.   {Thus, by the definition of $P$,}
$$\begin{aligned} P({g}\cdot Q) &=\frac{(1+\alpha)RN(Q)R^T-I}{\alpha-1}\\ &=R\left(\frac{(1+\alpha)N(Q)-I}{\alpha-1}\right)R^T\\ &=RP(Q)R^T. \end{aligned}$$
{With this identity, the definition of the boost matrix takes the form}
$$\begin{aligned} A_{\alpha,\beta}({g}\cdot Q) &=sRP(Q)R^T+s^{-1}\big(I-RP(Q)R^T\big)\\ &=R\big(sP(Q)+s^{-1}(I-P(Q))\big)R^T\\ &=RA_{\alpha,\beta}(Q)R^T. \end{aligned}$$
Since $c({g}\cdot Q)={g}(c(Q))$, for every $x\in\R^2$ we have
$$\begin{aligned} f_{{g}\cdot Q}({g}(x)) &={g}(c(Q))+RA_{\alpha,\beta}(Q)R^T \big(rR(x-c(Q))\big)\\ &={g}(c(Q))+rRA_{\alpha,\beta}(Q)(x-c(Q))\\ &={g}(f_Q(x)). \end{aligned}$$
Thus $f_{{g}\cdot Q}\circ {g}={g}\circ f_Q$, and consequently
$$\Theta_{\alpha,\beta}({g}\cdot Q) ={g}\cdot\Theta_{\alpha,\beta}(Q).$$
This shows that the map $\Theta_{\alpha,\beta}$ is  equivariant under all similarities, and hence under $G$.

Finally, {by equivariance, $\Theta_{\alpha,\beta}$ induces} a bijection
$$\overline{\Theta}_{\alpha,\beta}: \Quad(\alpha)/G\longrightarrow\Quad(\beta)/G$$
whose inverse is induced by $\Theta_{\beta,\alpha}$.  By Lemma~\ref{asp inv}, the two aspect fibers are $G$-invariant.  Since they are embedded Nash submanifolds of $\Quad$, the argument at the beginning of the proof of Lemma~\ref{Nash quotient} shows that the ambient Nash, free, and proper $G$-action restricts to a Nash, free, and proper action on each fiber. Thus Lemma~\ref{Nash lemma} shows that their quotient maps are surjective Nash submersions.  If $\pi_\alpha:\Quad(\alpha)\to\Quad(\alpha)/G$ and $\pi_\beta:\Quad(\beta)\to\Quad(\beta)/G$ are the quotient maps, then
$$\overline{\Theta}_{\alpha,\beta}\circ\pi_\alpha =\pi_\beta\circ\Theta_{\alpha,\beta}.$$
The right-hand side is Nash, so Lemma~\ref{Nash descent} shows that $\overline{\Theta}_{\alpha,\beta}$ is Nash.  Applying the same argument to $\Theta_{\beta,\alpha}$ shows that its induced inverse is Nash.  This proves the theorem.
\end{proof}

By Theorem~\ref{TFAE}, the quadrangles with aspect in the interval $(0,1)$ {are the} {nonexceptional} quadrangles whose {center }loci are ellipses, and those with aspect in $(-1,0)$ {are the} {nonexceptional} quadrangles whose {center }loci are hyperbolas. We call the former {\it {elliptic}} {nonexceptional} quadrangles and the latter {\it hyperbolic} {nonexceptional} quadrangles. Theorem~\ref{hard} raises the question of whether the aspect map restricted to {elliptic nonexceptional} quadrangles is a fiber bundle. This is so, and the bundle is in fact trivial.

 \begin{corollary} Let $G$ be a Nash subgroup of $\operatorname{Sim}(\R^2)$. {Let $\Quad(+)$ and $\Quad(-)$ denote the sets of elliptic and hyperbolic nonexceptional quadrangles, respectively.} There are Nash diffeomorphisms
\begin{eqnarray*}
{\Quad(+)}/G & \simeq &
 \{{\mbox{\rm aspect } } 1/2 {\mbox{\rm{ quadrangles}}}\}/G
  \times \R,\\
{\Quad(-)}/G & \simeq &
  \{{\mbox{\rm aspect } } -1/2 {\mbox{\rm{ quadrangles}}}\}/G
  \times \R.
  \end{eqnarray*}
 \end{corollary}

	  \begin{proof}
{Choose} {\nobreak{${\eta}:(0,1) \rightarrow \R$}} to be any Nash diffeomorphism from $(0,1)$ onto $\R$, e.g.,
$${\eta}(t)=\frac{2t-1}{t(1-t)}.$$
The construction of $\Theta_{\alpha,\beta}$ in Theorem~\ref{hard} depends Nashly on $Q$, $\alpha$, and $\beta$ for $\alpha,\beta\in(0,1)$, and the aspect map is Nash on $\Quad(+)$ by Lemma~\ref{asp is nash}. Hence the map $$\phi:\Quad(+) \rightarrow \Quad(1/2) \times \R:Q \longmapsto (\Theta_{{\alpha}(Q),1/2}(Q), {\eta}({\alpha}(Q)))$$ is Nash. The proof of Theorem~\ref{hard} shows that $\Theta_{\beta,\alpha}=\Theta_{\alpha,\beta}^{-1}$ for every pair $\alpha,\beta$ of the same sign. Therefore, $\phi$ has inverse $$\Quad(1/2) \times \R \longrightarrow \Quad(+):(Q,s) \longmapsto \Theta_{1/2,{\eta}^{-1}(s)}(Q),$$ and so $\phi$ is a Nash diffeomorphism. Similarly, let ${\eta_-}:(-1,0)\to\R$ be the Nash diffeomorphism ${\eta_-}(t)={\eta}(-t)$. Then $$\psi:\Quad(-) \longrightarrow \Quad(-1/2)\times \R: Q\longmapsto (\Theta_{{\alpha}(Q),-1/2}(Q),{\eta_-}({\alpha}(Q)))$$ is a Nash diffeomorphism with inverse $$\Quad(-1/2)\times\R\longrightarrow \Quad(-):(Q,s)\longmapsto \Theta_{-1/2,{\eta_-}^{-1}(s)}(Q).$$ Since aspect is invariant under similarities and the proof of Theorem~\ref{hard} shows that the maps $\Theta_{\alpha,\beta}$ are equivariant, the maps $\phi$ and $\psi$ are $G$-equivariant when $G$ acts trivially on the $\R$-coordinate. Thus these diffeomorphisms descend to diffeomorphisms from the orbit spaces in the statement of the corollary onto $\Quad(1/2)/G\times\R$ and $\Quad(-1/2)/G\times\R$, respectively. Let $\overline{\phi}$ denote the map induced by $\phi$, and let $\pi_+:\Quad(+)\to\Quad(+)/G$ and $\pi_{1/2}:\Quad(1/2)\to\Quad(1/2)/G$ be the quotient maps. Then
$$\overline{\phi}\circ\pi_+ =(\pi_{1/2}\times\operatorname{id}_{\R})\circ\phi.$$
The right-hand side is Nash, and $\pi_+$ is a surjective Nash submersion, so Lemma~\ref{Nash descent} shows that $\overline{\phi}$ is Nash. The same argument applied to the inverse of $\phi$, and to $\psi$ and its inverse, shows that the two descended maps are Nash diffeomorphisms.
	 \end{proof}

In a sequel, we study these model fibers for the groups of orientation-preserving isometries and similarities using methods based on the theory of bisector fields developed in \cite{OWPencil,OWQuad,OWProj}.

\section*{Statements and Declarations}

\smallskip

{\it Funding.} The first author was supported by NSF grant DMS-2231414.

\smallskip

{\it Competing interests.} The authors have no competing interests to declare that are relevant to the content of this article.

 \smallskip

{\it Data availability.} No datasets were generated or analyzed during the current study.

\smallskip

{\it Technology and AI use.} During the preparation of this manuscript, the authors used ChatGPT   for checking and simplifying arguments. All mathematical   proofs and results however are the authors' own. Maple was used to assist with and check calculations.


\begin{thebibliography}{99}

\bibitem{ACGR} A. Aizenbud, S. Carmeli, D. Gourevitch, L. Radzivilovski, Quotients of Nash manifolds by equivalence relations, Arnold Math. J., to appear, {arXiv:2304.02287}.


\bibitem{AG} A. Aizenbud and D. Gourevitch, De-Rham theorem and Shapiro lemma for Schwartz functions on Nash manifolds, Israel J. Math. 177 (2010), 155--188.

\bibitem{ArrMuc} J.~A.~Arredondo and J.~Muci\~no-Raymundo, Plane polynomials and Hamiltonian vector fields determined by their singular points, Results Math. 79 (2024), article 105.

\bibitem{AM} M. Artin and B. Mazur, On periodic points, Ann. of Math. (2) 81 (1965), no. 1, 82--99.

\bibitem{BambergPenttila} J.~Bamberg and T.~Penttila, {\it Analytic Projective Geometry}, Cambridge University Press, Cambridge, 2023.

\bibitem{BCR} J. Bochnak, M. Coste and M.-F. Roy,  {\it Real Algebraic Geometry}, {Springer-Verlag}, Berlin / Heidelberg, 1998.

\bibitem{Casey} J.~Casey, {\it A Treatise on the Analytical Geometry of the Point, Line, Circle, and Conic Sections}, 2nd ed., revised and enlarged, Hodges, Figgis, \& Co., Dublin; Longmans, Green, \& Co., London, 1893.


\bibitem{Del} A.~Del Centina, Brianchon and Poncelet's joint memoir, the nine-point circle, and beyond, Arch. Hist. Exact Sci. 76 (2022), no. 4, 363--390.

\bibitem{FM} W.~Fulton and R.~MacPherson, A compactification of configuration spaces, Ann. of Math. (2) 139 (1994), no.~1, 183--225.

\bibitem{GN} C.~Gibson and P.~Newstead,  On the geometry of the planar 4-bar mechanism, Acta Appl. Math. 7 (1986), no. 2, 113--135.

\bibitem{GonLop1} A. Gonz\'alez and J. L.  L\'opez-L\'opez, Compactness of spaces of convex and simple quadrilaterals, Bull. Aust. Math. Soc. 94 (2016), no. 3, 507--521.

\bibitem{GonLop2} A. Gonz\'alez and J. L. L\'opez-L\'opez, Spaces of special quadrilaterals, Bull. Aust. Math. Soc. 100 (2019), no. 1, 155--167.

\bibitem{GonMen} A. Gonz\'alez and M. Sedano-Mendoza, Moduli space of simple polygons, {arXiv:2303.15406 (2023)}.

\bibitem{KapMil1} M. Kapovich and J. Millson, On the moduli space of polygons in the Euclidean plane, J. Differential Geom. 42 (1995), no. 1, 133--164.

\bibitem{KapMil2} M. Kapovich and J. Millson, The symplectic geometry of polygons in Euclidean space, J. Differential Geom. 44 (1996), no. 3, 479--513.


\bibitem{Lee} {J.~M.~Lee, {\it Introduction to Smooth Manifolds}, 2nd ed., Graduate Texts in Mathematics 218, Springer, New York, 2013.}

{\bibitem{Martin} J.~L.~Martin, Geometry of graph varieties, Trans. Amer. Math. Soc. 355 (2003), no.~10, 4151--4169. \par}

\bibitem{Mint} M.~Minthorn, The nine-point conic, master's dissertation, University of California, Berkeley, 1912.

\bibitem{NavSan} J.~A.~Navarro Gonz\'alez and J.~B.~Sancho de Salas, {\it $C^\infty$-Differentiable Spaces}, Lecture Notes in Mathematics 1824, Springer-Verlag, Berlin, 2003.

\bibitem{OWPencil} B.~Olberding and E.~A.~Walker, Bisector fields and pencils of conics, Aequ. Math. 99 (2025), no. 2, 351--376.

\bibitem{OWQuad} B.~Olberding and E.~A.~Walker, Bisector fields of quadrilaterals, Beitr. Algebra Geom. 66 (2025), no. 1, 135--157.

\bibitem{OWProj} B. Olberding and E. A. Walker, Geometry of lines asymptotic to a pencil of conics, Aequ. Math. 99 (2025), no. 5, 2423--2456.

\bibitem{OWEnt}  B.~Olberding and E.~A.~Walker, Entanglement of four-point configurations, submitted.

\bibitem{Palais}  R.~S.~Palais, On the existence of slices for actions of non-compact Lie groups, Ann. of Math. (2) 73 (1961), no.~2, 295--323.




\bibitem{ShiVan} D. Shimamoto and C. Vanderwaart, Spaces of polygons in the plane and Morse theory, Amer. Math. Monthly 112 (2005), 289--310.


\bibitem{StClair} J. {L.} StClair, Geometry of spaces of planar quadrilaterals, PhD dissertation, Virginia Polytechnic Institute { and State University}, 2011.

\bibitem{Sun} B.~Sun, Almost linear Nash groups, Chinese Ann. Math. Ser. B 36 (2015), no.~3, 355--400.

\bibitem{Vac} M. Vaccaro, Historical origins of the nine-point conic---the contribution of Eugenio Beltrami, Historia Math. 51 (2020), 26--48.

\end{thebibliography}
\end{document}